\UseRawInputEncoding
\documentclass[11pt]{article}
\usepackage{authblk}

\usepackage[T1]{fontenc}
\usepackage[utf8]{inputenc}
\usepackage{hyperref}
\usepackage{mathtools, amsmath,amsfonts,amssymb,amscd,amsthm,amsbsy,amsxtra,bbm,bm, epsf,calc,comment,appendix}
\usepackage{xcolor}
\usepackage{datetime}

\usepackage[numbers,sort&compress]{natbib}
\usepackage{latexsym}
\usepackage[english]{babel}
\usepackage{enumerate}
\usepackage{graphicx}
\usepackage{epsfig}

\def\XXint#1#2#3{{\setbox0=\hbox{$#1{#2#3}{\int}$ }
\vcenter{\hbox{$#2#3$ }}\kern-.6\wd0}}

\newcommand \commentout[1] {}

\def\II{{\rm I\kern-0.5exI}}
\def\III{{\rm I\kern-0.5exI\kern-0.5exI}}

\newcommand{\norm}[1]{\lVert #1 \rVert}

\newcommand{\RR}{\mathbb{R}}

\newcommand{\ZZ}{\mathbb{Z}}

\DeclareMathOperator*{\essinf}{ess\, inf}

\DeclareMathOperator{\loc}{\textup{loc}}

\DeclareSymbolFont{bbold}{U}{bbold}{m}{n}
\DeclareSymbolFontAlphabet{\mathbbold}{bbold}

\newcommand{\vp}{\varphi}

\DeclareMathOperator{\sgn}{sgn}

\newcommand{\spt}{\textup{spt}}

\newcommand{\dom}{\textup{dom}}
\newcommand{\wnu}{w_{\nu}}
\newcommand{\pnu}{p_{\nu}}

\newcommand{\twnu}{\tilde{w}_{\nu}}
\newcommand{\tpnu}{\tilde{p}_{\nu}}
\newcommand{\trhonu}{\tilde{\rho}_{\nu}}

\newcommand{\rhonu}{\rho_{\nu}}

\newcommand{\dx}[1]{\mathop{}\!\mathrm{d} #1}

\numberwithin{equation}{section}
\newtheorem{theorem}{Theorem}[section]
\newtheorem{lemma}[theorem]{Lemma}
\newtheorem{prop}[theorem]{Proposition}
\newtheorem{cor}[theorem]{Corollary}

\newtheorem{assumption}[theorem]{Assumption}

\theoremstyle{remark}
\newtheorem{remark}[theorem]{Remark}

\theoremstyle{definition}
\newtheorem{definition}[theorem]{Definition}

\begin{document}
\title{On the singular limit of Brinkman's law to Darcy's law} 
\author[1]{Noemi David}
\author[2]{Matt Jacobs}
\author[3]{Inwon Kim}
\affil[1]{CNRS, LMRS, Université de Rouen}
\affil[2]{Department of Mathematics, UCSB}
\affil[3]{Department of Mathematics, UCLA}

\maketitle

\begin{abstract}
In this paper we study singular limits of congestion-averse growth models, connecting different models describing the effect of congestion. These models arise in particular in the context of tissue growth. The main ingredient of our analysis is a family of energy evolution equations and their dissipation structures, which are novel and of independent interest. This strategy allows us to consider a larger family of pressure laws as well as proving the joint limit, from a compressible Brinkman's model to the incompressible Darcy's law, where the latter is a Hele-Shaw type free boundary problem.
\end{abstract}

\medskip


\section{Introduction}
 We consider the singular limit of a PDE model of congestion-averse tissue growth. This model 
 is based on \emph{Brinkman's law} with a viscosity parameter $\nu>0$.  In this setting, the tumor cell density $\rhonu$ evolves according to the equation
\begin{equation}\label{eq: rho}
\left\{\begin{split}
    \partial_t \rho_{\nu}-\nabla \cdot(\rho_{\nu} \nabla w_{\nu})&=\rho_{\nu}G(\pnu) ,\\[0.3em]
  -\nu\Delta \wnu +\wnu&=\pnu,
    \end{split}\right.
\end{equation}
in $\mathbb{R}^d\times (0,\infty)$, where $G$ is a growth term, and the pressure $\pnu$ is generated by the density and a convex function  $f_{\nu}:\RR\to\RR$, through the duality relation
\begin{equation}\label{eq: p}
    p_{\nu}\in \partial f_{\nu}(\rho_{\nu}).
    \end{equation} 
Namely, the pressure is generated to dissipate the total internal energy 
\begin{equation}\label{eq: energy}
    F_{\nu}(\rho):=\int_{\RR^d} f_{\nu} (\rho) \dx x.
    \end{equation}
    {We assume $f_{\nu}:\RR\to \RR\cup\{+\infty\}$ to be convex and monotone increasing, so that the energy penalizes concentration. We are particularly interested in functions $f_{\nu}$ that carry extremely high penalties after a critical threshold, which then become singular in the limit $\nu\to 0$. A particular example we consider is   \begin{equation}\label{pme}
    f_{\nu}(\rho) = \frac{1}{\gamma_{\nu}}\rho^{\gamma_\nu},
    \end{equation}
    where $\gamma_{\nu}\to\infty$ as $\nu\to 0$, which enforces incompressibility in the limit.    
The growth term $G$ is assumed to be non-increasing with bounded positive set in $(0,\infty)$. Our monotonicity assumptions on $f_{\nu}$ and $G$ describe congestion-averse growth: high density values yield high pressure values through \eqref{eq: p}, which then decreases $G$ and thus slows down the production of new density. We refer to Section~\ref{sec: assumptions} for a full set of assumptions on $f_{\nu}$, $G$ and the initial data $\rho_{\nu}(\cdot,0)$.

 When $\nu=0$, \eqref{eq: rho} is an evolution equation following {\it Darcy's law
 } with pressure-dependent growth:
 \begin{equation}\label{eq:limit}
 \rho_t - \nabla\cdot(\rho\nabla p) = \rho G(p), \quad p\in \partial f_0(\rho).
 \end{equation}
 In the context of tumor growth, Brinkman's law with a viscosity coefficient $\nu>0$ extends  Darcy's law  by incorporating the effect of friction between cells \cite{PV2015, DEBIEC2020, DeSc}. 

 \smallskip
 
The aim of this paper is rigorously establishing the convergence from \eqref{eq: rho} to \eqref{eq:limit} as  $\nu\to 0$, for a general family of convex functions $\{f_{\nu}\}_{\nu\geq 0}$. As $\nu\to 0$, the order of the equation degenerates, and in the particular example we consider in \eqref{pme}, the regularity of $f_{\nu}$ also degenerates as $\gamma_{\nu}\to\infty$.  Indeed, this represents a pressure law transitioning from a soft-congestion penalty (compressible) to a hard-congestion constraint (incompressible). Our aim is to produce a flexible framework, which can simulataneously handle both of the singular transitions that arise as $\nu\to 0$.

\smallskip

In essence, the central difficulty in studying the limit $\nu\to 0$ for singular $f_{\nu}$, lies in obtaining sufficient compactness properties for $\rhonu$ and $\pnu$. Indeed, for ill-behaved $f_{\nu}$, it is difficult to transfer the time regularity properties of $\rhonu$  to $\pnu$, or conversely, to transfer the spatial regularity properties of $\pnu$ to $\rhonu$. As a result, standard a priori estimates do not imply the strong compactness of either $\rhonu$ or $\pnu$.  This lack of compactness  defeated previous approaches, as one cannot verify that the nonlinear terms $\rhonu\nabla \wnu, \,  \rhonu G(\pnu)$ converge weakly to their appropriate limits $\rho\nabla p, \, \rho G(p)$.

In this paper, we overcome this challenge via a newfound family of {\it energy evolution equations (EEE) } associated to Brinkman's law, which hold for almost any reasonable convex energy of the density: see Theorem~\ref{main:eee} for further discussions of this intriguing result. Using these evolution equations, we are able to establish several new powerful estimates previously unavailable in the literature (most crucially the estimate in Proposition \ref{prop:w_gradient_control}), allowing us to pass to the limit in the nonlinear terms.

\smallskip

An important consequence of our result addresses the case \eqref{pme}, 
where the limiting energy $f_0$ is given by \begin{equation}
        \label{graph}
          f_{0}(a)= 0 \quad \hbox{ if } a\leq 1, \quad  \quad\hbox{otherwise } +\infty,
          \end{equation}
          which has seen a recent surge of interest.
{Section~\ref{sec: intro ref} describes available results and different approaches taken to study models with the choice of $f_{\nu}$ as \eqref{pme}, respectively either for $\nu>0$ and $\nu=0$ or the limit $\nu\to 0$ with fixed nonlinearity $\gamma_{\nu} \equiv \gamma >0$. Our result establishes a rigorous link between these models for the first time.}
 
 \smallskip
  
{Even though our analysis is motivated by the above specific example, we emphasize that our approach is rather general and covers a general family of  $f_{\nu}$.} We can address for instance  
    $${f_\nu(a)=-\nu (a+\ln(1-a)), \quad f'_\nu(a)=\nu \frac{a}{1-a}} \hbox{ for } 0<a<1, \quad\hbox{ otherwise } +\infty,$$
    which are alternative popular choices in congestion-averse tissue growth models \cite{hecht2017incompressible, DHV}, or the singular diffusion equations 
  {with non-differentiable internal energy} 
    that are considered in \cite{KM2021} (see Section~\ref{sec: assumptions}).

We first state a rough version of the EEE (see Proposition~\ref{prop:eee} for a full statement), which is the central ingredient of our analysis. 

\begin{theorem}\label{main:eee} 
For a large class of convex functions  $e_\nu:{\rm dom}(f_{\nu})\to \mathbb{R}$ of the density, there is a convex function $z_{\nu}:\mathbb{R}\to \mathbb{R}$ of the pressure, such that 
\begin{equation}\label{eq:eee_intro}
\partial_t e_{\nu}(\rhonu)-\nabla \cdot \big(e_{\nu}(\rhonu)\nabla \wnu\big)-\Delta z_{\nu}(\wnu)+{\mathcal{L}_z} =\rhonu e_{\nu}'(\rhonu)G(\pnu)
\end{equation}
holds in the weak sense, where 
$$
\mathcal{L}_z:= z''_{\nu}(\wnu)|\nabla \wnu|^2+\big(z'_{\nu}(\wnu)-z'_{\nu}(\pnu)\big)\frac{\wnu-\pnu}{\nu} \geq 0.
$$
\end{theorem}
The above equation can be interpreted as a continuity equation for $e_{\nu}$ along the Brinkman vector field $\nabla \wnu$, with correction terms $\Delta z_{\nu}(\wnu)$ measuring local redistribution of energy through diffusion, and $\mathcal{L}_z$ representing the local dissipation of the energy to friction. The flexibility of this dissipation structure over the choice of $e_{\nu}$, is rather surprising and holds broader interest beyond this particular setting.  {For example our energy flow approach appears to apply to systems in a mostly parallel way: we briefly sketch out the argument with the derivation of the core energy evolution equations for systems in Appendix~\ref{app: sys}.}

  \medskip
  
   We now state our convergence result established in Section~\ref{sec: strong comp}. First, we make some natural assumptions on the initial data of the approximate problems.
\begin{definition}\label{def in}
      For a given nonnegative $\rho^{\mathrm{in}}\in L^1\cap L^{\infty} (\mathbb{R}^d)$, we say that  a family of nonnegative initial data $(\rhonu^{\mathrm{in}})_{\nu>0} $, uniformly bounded in $L^1\cap L^{\infty}(\mathbb{R}^d)$, are {\it well-prepared} if 
     $\rho^{\mathrm{in}}_\nu \rightharpoonup \rho^{\mathrm{in}}$ in $L^1(\mathbb{R}^d)$,    
   
and there exists $B\geq 0$ such that $\rhonu^{\textup{in}}\leq \sup\partial f_{\nu}^*(B)$ almost everywhere.
\end{definition}
     
As we will explain in Section~\ref{sec: strategy}, several versions of Theorem~\ref{main:eee} are applied to establish the following:

\begin{theorem}\label{theorem}
    Let $f_{\nu}$ and $G$ satisfy Assumptions \ref{as: reaction}-\ref{as: energies}. For a given  $\rho^{\mathrm{in}}\in L^1\cap L^{\infty} (\mathbb{R}^d)$, let $\{\rho_\nu\}_{\nu> 0}$ be a sequence of weak solutions to \eqref{eq: rho}  with well-prepared initial data $\{\rho^{\textup{in}}_\nu\}_{\nu> 0}$. Then, there exist  $\rho\in 
L_{\mathrm{loc}}^\infty([0,\infty);L^1(\RR^d)\cap L^\infty(\RR^d)),$ 
  $p\in L^{\infty}([0,\infty)\times\RR^d)$ 
such that
    \begin{align*}
         \rhonu &\rightharpoonup \rho, \text{ weakly in } L^2_{\loc}([0,\infty); L^2(\RR^d)),\\[0.5em]
         \pnu &\to p,\text{ strongly in } L^2_{\loc}([0,\infty)\times \RR^d),\\[0.6em]
     \nabla \wnu &\to \nabla p, \text{ strongly in } L^2_{\loc}((0,\infty); L^2(\RR^d)),
    \end{align*}
 and $(\rho, p)$ is a solution to \eqref{eq:limit}  with initial data $\rho^{\textup{in}}$.  
\end{theorem}
 We use the standard notion of weak solutions for \eqref{eq: rho} and \eqref{eq:limit} for divergence form equations, see Definitions~\eqref{def:wk_sol_brinkman} and ~\ref{def:wk_sol_HS}.

\subsection{Singular limits in congestion models: literature}\label{sec: intro ref}
As mentioned above, the specific case of the incompressible energy - $f_0$ given in \eqref{graph} -  has gained significant interest in recent literature, for bio-medical applications (tumor growth) as well as for models of pedestrian motion (see for instance \cite{DM2016, MBSV2011, MS2016}).
{The solution of \eqref{eq:limit} with this choice of $f_0$} is usually obtained as the incompressible limit of equations of porous-medium type, namely of \eqref{eq: rho} with $\nu=0$ and pressure-density couplings given by power laws $$p=f'_\gamma(\rho)=\rho^\gamma, \quad \gamma\geq 1.$$ As $\gamma\to \infty$ the pressure converges to the incompressible graph~\eqref{graph}, and the limit equation~\eqref{eq:limit} in this case corresponds to a solution of the so-called Hele-Shaw problem. This limit has attracted a lot of interest and a vast literature is now available. The incompressible limit has also been studied for \eqref{eq: rho} with $f_{\nu}$ given as in \eqref{pme} with fixed $\nu>0$ and $\gamma\to\infty$. This limit (to the incompressible Brinkman problem) generates a free boundary problem that moves with a nonlocal dependence on the pressure. Due to the nonlocality, the pressure is less regular here than in the Hele-Shaw flow. In particular, the pressure for the incompressible Brinkman problem is not continuous across the free boundary \cite{PV2015, DEBIEC2020, DeSc}, whereas generic continuity is expected for the Hele-Shaw flow.

\medskip

In this paper, we answer two open questions in the context of these singular limits; we prove that the diagram below commutes, namely that the incompressible Brinkman problem converges to the Hele-Shaw problem, and that it is possible to pass to the joint limit and directly obtain the Hele-Shaw problem from the compressible Brinkman problem \eqref{eq: rho}.  Namely, our work fills in the red arrows, which were not previously known.

\medskip

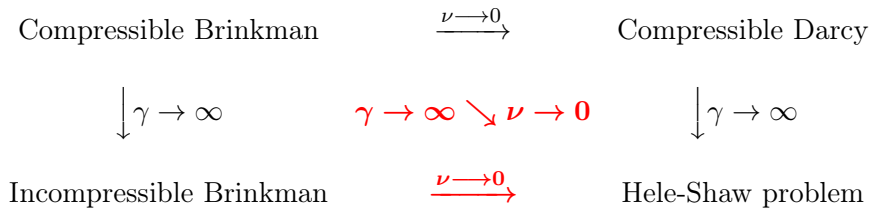
\begin{equation*}
\begin{array}{cccc}
    &\text{Compressible Brinkman} &{\xrightarrow[]{\nu\longrightarrow 0}} &\text{Compressible Darcy}\\[1.3em]
    &\Big\downarrow{\gamma\to \infty} &{\color{red}\boldsymbol{ \gamma\to\infty\searrow \nu \to 0}}   &\Big\downarrow{\gamma\to \infty}\\[1.3em]
    &\text{Incompressible Brinkman} &{{\color{red}\boldsymbol{\xrightarrow[]{\nu\longrightarrow 0}}}} &\text{Hele-Shaw problem}
\end{array}
\end{equation*}

\bigskip

We discuss the main strategies applied to the previous studies of the three singular limits presented in the diagram above and why they cannot be applied to our problem.

\begin{itemize}
\item $\nu=0, \gamma \to \infty$: as mentioned above, this limit is the most studied one and it bridges the gap between the porous medium equation and the Hele-Shaw problem. The first result in the context of tumor growth models, \cite{PQV}, considers a porous medium equation with a pressure-dependent growth rate. 
The strategy adopted to pass to this limit involves some classical tools of the theory of the porous medium equation, such as the Aronson-B\'enilan estimate. Let us also recall that this equation guarantees the propagation of uniform $BV$-bounds, hence strong compactness of pressure and density are immediately proven. 
Subsequently, research branched out in many different directions, and models including the presence of nutrients \cite{DP21}, local and nonlocal external drifts \cite{AKY14, CKY18, KPW19, CCY2019} have been analyzed. New methods have been applied, such as extensions of the Aronson-B\'enilan estimate in $L^p$ spaces \cite{GPS}, and strategies relying on its gradient-flow structure \cite{LX2021, Dav23, Jac21, jacobs_lagrangian}.  
The latter have been particularly useful in treating multi-species systems for which strong compactness of the single species is not directly available. Here, the strong compactness of the velocity field, namely $\nabla p$, is obtained by exploiting the dissipative structure of the total density equation. Let us also remark that viscosity solutions methods, \cite{KP17, KPW19, KPS}, have also been widely applied to this problem, based on the underlying comparison principle. The same limit is also well-known in the context of crowd motion models, for which is usually referred to as \textit{hard congestion limit}, we refer the reader to \cite{AKY14, DMSV2016, MBSV2011} and references therein.   Note however, that many of these techniques strongly rely on the specific structure of Darcy's law and cannot be applied to Brinkman's law even when $\nu$ is very small.
    \item $\nu>0, \gamma\to \infty$: the incompressible limit for the Brinkman problem presents technical difficulties in showing the strong compactness of the pressure. While $w$ is regular due to the elliptic equation that it satisfies, only uniform bounds are available on pressure. 
    To show the strong compactness of the pressure, it is necessary to use a representation through Young's measures and show that the sequence does not oscillate, see \cite{PV2015}. This problem has also been studied for multi-species systems in \cite{DeSc, DEBIEC2020}, where the authors show strong compactness of the single densities through arguments \textit{\'a la Belgacem-Jabin} \cite{BJ2013, BJ2017, BJ2018}. 
 We stress that this argument relies strongly on the elliptic regularity associated to the velocity field $\nabla w$ and thus cannot be applied in the limit $\nu\to 0$.  
    \item $\gamma\geq 1, \nu\to 0$: this problem was first addressed for linear pressure laws, $\gamma=1$, in \cite{DDMS}, whose strategy was then extended in \cite{ES2023} for power pressure laws with $\gamma>1$. The important ingredient is the energy dissipation \textit{equality} associated with the limit problem. The authors then combine it with an entropy dissipation \textit{inequality} of the approximate problem, to obtain 
     strong compactness of $\nabla \wnu$.
     This strategy builds on a preliminary step that yields strong compactness of the density $\rhonu$ and the pressure $\pnu$. For fixed $\gamma$, this was achieved by the \textit{uniform} boundedness of $\partial_t \pnu$ in some negative Sobolev space.  However, the upper bound used in \cite{ES2023} blows up as $\gamma\to \infty$. Indeed, obtaining any uniform bound on $\partial_tp$ in this limit remains an elusive task. For our analysis of the joint limit, we remove this preliminary step and show that all the necessary compactness can be obtained by a family of energy evolution equalities. We will explain this key part of our analysis in the next section.
\end{itemize}

\noindent

\subsection{Strategy}
\label{sec: strategy}
Let us briefly describe the main challenge in establishing the convergence result. While the velocity field $\nabla w$ enjoys a higher regularity than $\nabla p$ for $\nu>0$, it is difficult to obtain strong compactness of the density or pressure variable as $\nu$ tends to zero from \textit{a priori} estimates, especially since the internal energy converges to the incompressible one as $\nu\to 0$. Sending $\nu\to 0$ in the weak convergence limit,  the continuity equation we obtain is of the form
$$
\partial_t\rho - \nabla \cdot m = R,
$$
where $m$ and $R$ remains to be identified.

\medskip

In Section~\ref{sec: strong comp}, we will use different versions of the EEE \eqref{eq:eee_intro} to endow uniform regularity properties to $w_{\nu}$ in the process of characterizing both $m$ and $R$. As we illustrate briefly below, different choices of the energy $e_{\nu}$ or $z_{\nu}$ in the EEE will generate different forms of regularity statements in terms of the corresponding dissipation structure. 

\medskip

First we will show that $m=\rho\nabla p$.
Observe that the corresponding flux term for  $\rho_{\nu}$, from the continuity equation  \eqref{eq: rho}, is $\rho_{\nu}\nabla w_{\nu}$. 
Recalling that $p\in \partial f_0(\rho)$, if we let $f^*$ denote the convex conjugate of function $f$, then we have the equivalent relation $\rho\in \partial f_0^*(p)$.  Hence, $\rho\nabla p$ can be written as $\nabla f_0^*(p)$, thus, the central step in the characterization of $m$ is to show that $\rho_{\nu}\nabla \wnu$ is almost $\nabla f_{\nu}^*(w_{\nu})$, namely that 
$$
\rho_{\nu}\nabla w_{\nu}-\nabla f_{\nu}^*(w_{\nu}) = (f_\nu^{*\prime}(\pnu) - f_\nu^{*\prime}(\wnu))\nabla \wnu \rightharpoonup 0 \hbox{ as } \nu\to 0.
$$

Since $\wnu$ is formally expected to  converge to the same limit as $\pnu$,  the above calculation looks promising. But an important challenge arises since our $f_{\nu}^*$ is not uniformly $C^1$. Thus, there may be (albeit small, due to convexity) regions where its derivative may jump.
This is where we first use the EEE in a crucial way to yield a type of uniform regularity property for $\wnu$. This property allows us to ignore the sets where $f_\nu^{*\prime}(\wnu)$ has steep changes when integrating against $\nabla \wnu$. More precisely, we can use the EEE to quantitatively show that $\nabla\wnu$ does not concentrate on a set where $\wnu$ maps to a small-measured set of images.
See Sections~\ref{sec: m} and \eqref{est:uniform} for further discussion of this statement, as well as the rest of the convergence proof (for instance controlling $\pnu-\wnu$), where other types of EEE have been applied.

\medskip

It remains to show that $R=G(p)$. We will in fact show a stronger result in achieving this, namely we will show the strong convergence of the velocity field $\nabla\wnu$ to $\nabla p$.  As in \cite{SS} we compare the energy dissipation structure of the approximating flow with the limiting flow. On the other hand our approximate flow $\rho_{\nu}$ is not a gradient flow, and thus their argument cannot be directly applied. Instead our key idea here is to employ the $H^{-1}$ energy dissipation equation generated again by EEE (Corollary~\ref{cor:h-1}), where the corresponding dissipation is generated by $|\nabla f_{\nu}^*(w_{\nu})|^2$, to compare with the corresponding term in the limit equation, $|\nabla f_0^*(p)|^2$.

\medskip

Let us re-emphasize that the generalized energy dissipation structure that we obtain for $\rho_{\nu}$ is of independent interest. Its physical meaning in terms of the energy flows as well as its full usage remain open to be explored. 

\medskip

\paragraph{Outline of the paper.}  In the next section we present the main assumptions and the definitions of solutions for \eqref{eq: rho} and \eqref{eq:limit}. In Section~\ref{sec: a priori} we present some standard \textit{a priori} estimates on density and pressure, while Section~\ref{sec: eee} is devoted to the presentation and proof of the family of entropy dissipation inequalities. These will be employed in the last section, Section~\ref{sec: strong comp}, to characterize the limit flux and show the strong compactness of the velocity field, concluding the proof of the main theorem.

\paragraph{Acknowledgments} M.J. is supported by NSF-DMS-2400641.
 I. K. is supported by NSF-DMS-2153254. N.D. would like to acknowledge the hospitality of the Departments of Mathematics of UCLA and UCSB during her research stay, where the foundation for this project was laid. N.D. would also like to thank Filippo Santambrogio for generously supporting her visit to California with his ERC Advanced Grant EYAWKAJKOS.
\section{Preliminaries}\label{sec: assumptions}
Before presenting the main results of the paper, we state our assumptions on $G$ and the energy density functions $\{f_{\nu}\}_{\nu\geq 0}$.

\begin{assumption}\label{as: reaction}
Let $G\in C([0,+\infty))$ be a decreasing function such that $G(0)>0$.  Moreover, there exists $p_H>0$ such that $G(p_H)=0$.
\end{assumption}

\begin{assumption}\label{as: energies}
Let $f_{\nu}:\RR\to\RR\cup\{+\infty\}$ such that $f_{\nu}\to f_0$ a.e. in $\RR$ with the following property for each $\nu\geq 0$:
\begin{enumerate}
\item \label{as:energy_convex} $f_{\nu}$ is proper, lower semicontinuous, and convex, in $\mathbb R$. 

\item \label{as:energy_domain} there exists $a_0>0$ such that $[0,a_0]\subset \dom(f_{\nu}):=\{a\in \RR:\; \partial f_{\nu}(a) \neq \varnothing\}$.

\item\label{as: limit/a}  $f_{\nu}(a)=+\infty$ for all $a<0$, 
$\lim_{a\to 0^+} a^{-1}f_{\nu}(a)=0$, and $\lim_{a\to +\infty} a^{-1}f_{\nu}(a)=+\infty$.

\end{enumerate}
\end{assumption}

\begin{remark}
    Assumption~\ref{as: energies}  \eqref{as:energy_domain} requires the energies to be finite on $[0,a_0]$ for all $\nu\geq 0$. Hence, for any $M>0$ there exists a density $\rho:\RR^d\to [0,\infty)$ with total mass $M$, finite total energy, and bounded support. This gives a uniform way to eliminate degenerate energies where nonzero total mass and finite energy are incompatible.
\end{remark}
\begin{remark}
    Note that Assumption~\ref{as: energies}  \eqref{as: limit/a} ensures that the energy densities are identically infinity on negative values, which is natural in the context of density evolution equations where negative mass should be impossible.  The advantage of adopting this convention is that we will then automatically know that the dual functions $f_{\nu}^*$ are nonnegative and nondecreasing on all of $\RR$.  The remaining assumptions guarantee that the pressure can always be taken to be an element of $[0,\infty)$ and that the dual energy $f_{\nu}^*$ does not take the value $+\infty$ on $[0,\infty)$. 
\end{remark}
\begin{remark}
  Due to convexity of $f_{\nu}$, pointwise convergence assumption implies uniform convergence of both $f_{\nu}$ to $f_0$ and its convex dual $f^*_{\nu}$ to $f_0^*$   on compact subsets of $\dom(f_0)$ and $\dom(f_0^*)$ respectively, see, for instance, \cite{rockafellar}. 
\end{remark}

\medskip 

We now introduce the definitions of weak solutions to the two equations.

\begin{definition}[Weak solution for compressible Brinkman]
    \label{def:wk_sol_brinkman}
	We say that a nonnegative triple of functions  $(\rhonu, \pnu,\wnu)$ is a weak solution to \eqref{eq: rho} with initial data $\rhonu^{\mathrm{in}}\in L^1(\RR^d)\cap L^\infty(\RR^d)$ satisfying Definition~\ref{def in}, if $\rhonu\in L^{\infty}_{\loc}([0,\infty);L^1(\RR^d)\cap L^{\infty}(\RR^d))$, $\pnu\in L^\infty([0,\infty)\times\RR^d)$, $\wnu\in L^2_{\loc}([0,\infty);H^1(\RR^d))$,  we have the relations 
    \begin{align*}
	\pnu\in \partial f(\rhonu),\quad  	 - \nu \Delta w_\nu+w_\nu  = \pnu, \quad \textup{almost everywhere in} \;\;(0,\infty)\times \RR^d
	\end{align*}
   and for any test function $\varphi \in C_{c}^{\infty}([0,\infty); C^{\infty}(\RR^d ))$, $(\rhonu, \pnu, \wnu)$ satisfy the weak equation
	\begin{align*}
		-\int_0^{\infty}\!\!\! \int_{\RR^d} &\rhonu \partial_t \varphi dx dt 
        + \int_0^{\infty} \!\!\! \int_{\RR^d} \rhonu \nabla \wnu \cdot \nabla \varphi dx dt
		= \int_0^{\infty} \!\!\! \int_{\RR^d} \varphi \rhonu G(\pnu ) dx dt +\int_{\RR^d} \varphi(x,0)\rhonu^{\mathrm{in}}(x) dx.
	\end{align*}
    \end{definition}

\begin{definition}[Weak solution for incompressible Darcy]\label{def:wk_sol_HS}
    We say that a nonnegative pair of functions $(\rho, p)$ is a weak solution to \eqref{eq:limit} with nonnegative initial data $\rho^{\mathrm{in}}\in L^1(\RR^d)\cap L^\infty(\RR^d)$ such that $f_0(\rho^{\textup{in}})\in L^1(\RR^d)$, if $\rho\in L^{\infty}_{\loc}([0,\infty);L^1(\RR^d)\cap L^{\infty}(\RR^d))$, $p\in \partial f_0 (\rho)$, $p\in L^2_{\loc}([0,\infty);H^1(\RR^d))$ and there holds
	\begin{align*}
		-\int_0^{\infty}\!\!\! \int_{\RR^d} &\rho \partial_t \varphi dx dt 
        + \int_0^{\infty} \!\!\! \int_{\RR^d} \rho \nabla p \cdot \nabla \varphi dx dt
		= \int_0^{\infty} \!\!\! \int_{\RR^d} \varphi \rho G(p ) dx dt +\int_{\RR^d} \varphi(x,0)\rho^{\mathrm{in}}(x) dx.
	\end{align*}
	for any test function $\varphi \in C^{\infty}_c([0,\infty);C^{\infty}(\RR^d))$.
\end{definition}

\section{A Family of Energy Evolution Equations}
\label{sec: eee}

In this section, we will state a family of energy evolution equations for Brinkman's law, which will play a fundamental role in the rest of our analysis. Given a convex energy function $e_{\nu}$ defined on the density variable, the equation describes the evolution of the energy density $e_{\nu}(\rhonu)$ along the flow. In Section \ref{ssec:eee}, we will show in  Proposition~\ref{prop:eee} that, for a large class of energies $e_{\nu}$, there exists a convex function $z_{\nu}:\RR\to\RR$ such that the equation
\begin{equation}\label{eq:formal_energy_equation}
\partial_t e_{\nu}(\rhonu)-\nabla \cdot \big(e_{\nu}(\rhonu)\nabla \wnu\big)-\Delta z_{\nu}(\wnu)+{\mathcal{L}_z} =\rhonu e_{\nu}'(\rhonu)G(\pnu),
\end{equation}
is satisfied in a weak sense, where $\mathcal{L}_z$ is a nonnegative dissipation term
$$
\mathcal{L}_z:= z''_{\nu}(\wnu)|\nabla \wnu|^2+\big(z'_{\nu}(\wnu)-z'_{\nu}(\pnu)\big)\frac{\wnu-\pnu}{\nu} \geq 0.
$$
For the purposes of our discussion, we will assume for simplicity that $e_{\nu}$ and $z_{\nu}$ are sufficiently smooth on their respective domains; however, this will be an important issue in our subsequent analysis.

Crucially, $\mathcal{L}_z$ is nonnegative as a consequence of the convexity of $z_{\nu}$. Indeed, $z''_\nu$ is positive and $z'_\nu$ is order preserving, so both terms in $\mathcal{L}_z$ are nonnegative. In addition, the second term in $\mathcal{L}_z$ provides some control on $\sqrt{\nu}\Delta \wnu$; applying the fundamental theorem of calculus yields
\begin{align*}
\big(z'_{\nu}(\wnu)-z'_{\nu}(\pnu)\big)\frac{\wnu-\pnu}{\nu}&=\frac{|\wnu-\pnu|^2}{\nu}\int_0^1 z''(\beta \wnu+(1-\beta)\pnu)\, d\beta\\
&=\nu|\Delta \wnu|^2\int_0^1 z''(\beta \wnu+(1-\beta)\pnu)\, d\beta.
\end{align*}
Thanks to this nonnegativity, $\mathcal{L}_z$ has a physical interpretation as an instantaneous rate of energy dissipation, accounting for the loss of energy to friction along the flow.  

Let us emphasize that the dissipation structure of \eqref{eq:formal_energy_equation} was not previously known in the literature except in one very special case (see Section \ref{ssec:heuristics}).  Indeed, if one attempts to derive \eqref{eq:formal_energy_equation} by formally differentiating $e_{\nu}(\rhonu)$ in time, the dissipation structure $\mathcal{L}_z$ is not at all apparent --- an initial computation of $\frac{d}{dt} e(\rhonu)$ yields a much more mysterious term whose sign appears to be ambiguous (c.f. the term $M_{\nu}$ in equation \eqref{eq:formal_energy_equation_unclear}). In Section \ref{ssec:heuristics}, we will illustrate how we circumvent this ambiguity through the introduction of the convex function $z_{\nu}$, obtained via a non-trivial transformation of $e_{\nu}$.

\medskip

As we will see, the dissipation structure provided by $\mathcal{L}_z$, for a number of different choices of $e_{\nu}$, will play a fundamental role in our argument.
First, for a certain choice of $e_{\nu}$, one has $z''_{\nu}\equiv 1$, and thus,  \eqref{eq:formal_energy_equation} leads to uniform bounds on $|\nabla w_{\nu}|^2$ and $\nu|\Delta \wnu|^2$ with respect to $\nu$, see \eqref{e.EDI z''=1}. This provides some crucial estimates needed to pass to the limit $\nu\to 0$ in Section~\ref{sec: strong comp}. Next, we utilize the generality of the EEE in Section~\ref{sec: m}, where we will show that $\nabla w_{\nu}$ gives little mass to values of $\wnu$ where $f^{*\prime}_{\nu}(w_{\nu})$ has large slopes, which is crucial in characterizing the limit flux: here is the place where the general nature of our $e_{\nu}$ really stands out. Lastly, in Section~\ref{sec: final},  we exploit the dissipation structure for $e_{\nu}(a):=af_{\nu}(a)-2\int_0^a f_{\nu}(\alpha)\, d\alpha$ to obtain the strong compactness of $\nabla \wnu$, see Proposition~\ref{prop:nablafstar}. Notably, our strategy does not involve uniform bounds on any higher spatial derivatives or any time derivatives of $\wnu$, which are likely false in the most challenging cases such as $f_{\nu}(a)=\nu a^{1/\nu}$. 

 \medskip

Before we provide a rigorous derivation of \eqref{eq:formal_energy_equation}, we first give some brief heuristics to illustrate the main idea and explain the novelty of our contribution in the context of the previous literature.

\subsection{Heuristics}\label{ssec:heuristics}
A natural first approach to obtain the energy evolution equation is to take the time derivative of $e_{\nu}(\rhonu)$ and plug in the density evolution equation (\ref{eq: rho}).  Doing so,  one obtains the equation
\begin{equation}\label{eq:formal_energy_equation_unclear}
\partial_t e_{\nu}(\rhonu)-\nabla \cdot (e_{\nu}(\rhonu)\nabla \wnu)+ \underbrace{\big(e_{\nu}(\rhonu)-\rhonu e'_{\nu}(\rhonu)\big)\Delta \wnu}_{M_\nu}=\rhonu e'_{\nu}(\rhonu)G(\pnu).
\end{equation}
This is a continuity equation along the expected vector field $-\nabla \wnu$, with the additional source terms $M_{\nu}:=\big(e_{\nu}(\rhonu)-\rhonu e'_{\nu}(\rhonu)\big)\Delta \wnu$ and $\rhonu e'_{\nu}(\rhonu)G(\pnu)$. The growth term $\rhonu e'_{\nu}(\rhonu)G(\pnu)$ is fairly straightforward and can be handled via Gronwall's inequality. However, the additional term $M_{\nu}$ is much more mysterious.

In a very special case where $e_{\nu}(\rhonu)-\rhonu e'_{\nu}(\rhonu)$ simplifies,  previous literature on Brinkman's law was able to identify a dissipation structure in $M_{\nu}$ \cite{DDMS, ES2023}. First, in \cite{DDMS}, the authors considered Brinkman's law with the density-pressure coupling $\pnu=\rhonu$ (corresponding to the choice $f_{\nu}(a)=\frac{1}{2}a^2$) and studied the energy evolution equation for the entropy function $e_{\nu}(a)=a\log(a)-a$.  With these specific choices, one finds that 
\begin{equation}\label{eq:simple_case}
e_{\nu}(\rhonu)-\rhonu e'_{\nu}(\rhonu)=-\rhonu=-\pnu.
\end{equation} As a result, $M_{\nu}=-\pnu\Delta \wnu$ and one can then see that 
\[
-\pnu\Delta \wnu=-\wnu\Delta \wnu+(\wnu-\pnu)\Delta \wnu=-\frac{1}{2}\Delta (\wnu^2)+|\nabla\wnu|^2 +\nu|\Delta \wnu|^2,
\] 
where the second equality follows from the identity $\wnu-\pnu=\nu\Delta\wnu$.  This was extended in \cite{ES2023}, to the more general power law density-pressure coupling, $\pnu=\rhonu^{\gamma}$ for $\gamma>1$, (which corresponds to $f_{\nu}(a)=\frac{1}{\gamma+1}a^{\gamma+1}$), where it was noticed that by choosing $e_{\nu}(a)=\frac{a^{\gamma}-a}{\gamma-1}$, one obtains the same simplification $e_{\nu}(\rhonu)-\rhonu e'(\rhonu)=-\pnu$. Let us emphasize that these works only establish energy dissipation for a single choice of convex energy, which does not provide sufficient estimates for our strategy to work.

In contrast, we are able to show that the mysterious term $M_{\nu}$ is always a good term, as long as $e_{\nu}$ satisfies some mild technical conditions (see the assumptions in Lemma \ref{lem:z_from_e}).
The key insight in our work is that for each convex energy $e_{\nu}$ defined on the density, we can find a corresponding convex function $z_{\nu}:\RR\to\RR$ of the pressure such that
\begin{equation}\label{eq:density_to_pressure_convex_transformation}
\rhonu e'_{\nu}(\rhonu)-e_{\nu}(\rhonu)=z'_{\nu}(\pnu).
\end{equation}
Using the identity (\ref{eq:density_to_pressure_convex_transformation}), we are able to show that $M_{\nu}$ can always be written as the sum of a nonnegative term and a Laplacian. Indeed, we first write
\[ 
M_{\nu}:=\big(e_{\nu}(\rhonu)-\rhonu e'_{\nu}(\rhonu)\big)\Delta \wnu=-z'(\pnu)\Delta \wnu=(z'(\wnu)- z'(\pnu))\Delta \wnu-z'(\wnu)\Delta \wnu.
\]
After using the identities $\nu \Delta \wnu={\wnu-\pnu}$ and $\Delta (z(\wnu))=z''(\wnu)|\nabla \wnu|^2+z'(\wnu)\Delta \wnu$, we find 
\[
M_\nu=(z'(\wnu)-z'(\pnu))\frac{\wnu-\pnu}{\nu}-\Delta (z(\wnu))+z''(\wnu)|\nabla \wnu|^2. 
\]
Putting this all together, one finally obtains the energy evolution equation in the form \eqref{eq:formal_energy_equation}.

\subsection{Derivation of the energy evolution equation}\label{ssec:eee}

Now we turn towards rigorously establishing the energy evolution equation. 
As we saw above, the key step is the introduction of a convex function $z_{\nu}$ satisfying the identity \eqref{eq:density_to_pressure_convex_transformation}. In general, $e_{\nu}$ may not be differentiable and we need to replace $e'_{\nu}(\rhonu)$ by some element in $\partial e_{\nu}(\rhonu)$. To avoid worrying too much about the specific variables $\rhonu, \pnu$, we want to ensure that $e_{\nu}$ and $z_{\nu}$ are correctly linked for all possible values. Hence, we need to ensure that for all $a\in \dom(e_{\nu})$ and all $b\in \partial f_{\nu}(a)$, there exists $c\in \partial e_{\nu}(a)$ such that
\begin{equation}\label{eq:abc_equation}
ac-e_{\nu}(a)=z'_{\nu}(b).
\end{equation}
We begin with two Lemmas that explain how to construct the convex function $z_{\nu}$ from $e_{\nu}$ so that \eqref{eq:abc_equation} is satisfied (and vice-versa). 
While these constructions can be done starting from almost any convex $e_{\nu}$ (or $z_{\nu}$), the dissipation term $\mathcal{L}_z$ involves second derivatives of $z_{\nu}$. For this reason, we will restrict our attention to cases where the second derivatives of $z_{\nu}$ are sufficiently regular.   This will place some mild technical restrictions on the admissible choices of $e_{\nu}$.  

\medskip

\begin{lemma}\label{lem:z_from_e}
    Suppose that $e_{\nu}:\dom(f_{\nu})\to\RR$ is a convex function such that $e_{\nu}(0)=0$, and 
    \begin{equation}\label{eq:e_compatbility}
    \partial e_{\nu}(a)=\{S(b): b\in \partial f_{\nu}(a)\} \quad \textup{for all} \; a\in \dom(f_{\nu}),
    \end{equation}
    for some nondecreasing function $S\in W^{1,1}_{\loc}([0,\infty))$.  
If we define $z_{\nu}:[0,\infty)\to\RR$ by setting $z_{\nu}(0)=0$ and
\[
    z_{\nu}'(b)=\int_0^b f^{*\prime}_{\nu}(\beta)S'(\beta)\, d\beta,
    \]
    then $z_{\nu}$ is convex, $z_{\nu}\in W^{2,1}_{\loc}([0,\infty))$, and $z_{\nu}$ satisfies the coupling relation \eqref{eq:abc_equation} with $e_{\nu}$.
\end{lemma}
\begin{remark}
    Note that in general, $z_{\nu}$ is  not uniquely determined by $e_{\nu}$.   At singular points $a\in \dom(f_{\nu})$ where $\partial f_{\nu}(a)$ and $\partial e_{\nu}(a)$ both fail to be single-valued, one can replace $S$ by $S\circ R$ where $R:\RR\to\RR$ is any nondecreasing function that is the identity on the endpoints of $\partial f_{\nu}(a)$, without changing $e_{\nu}$.  On the other hand, if $R(b)\neq b$ for $b\in \textup{int}(\partial f_{\nu}(a))$,  then $z_{\nu}'$ will change. 
\end{remark}

\bigskip 

Now we have seen how to obtain $z_{\nu}$ from $e_{\nu}$.  Although we have been discussing the energy evolution equation from the perspective of $e_{\nu}$,  we will mostly be interested in the dissipation term $\mathcal{L}_z$ rather than properties of $e_{\nu}$.  Hence, it is often more useful to start from $z_{\nu}$ and define $e_{\nu}$ in terms of $z_{\nu}$. This is accomplished in the following Lemma.

\begin{lemma}\label{lem:e_from_z}
Suppose that $z_{\nu}:\RR\to\RR$ is a convex function such that $z_{\nu}(0)=z_{\nu}'(0)=0$ and $z_{\nu}\in W^{2,1}_{\loc}([0,\infty))$.  If we choose some $a_1\in \textup{int}(\dom(f_{\nu}))$ and define $e_{\nu}:\dom(f_{\nu})\to\RR$ by setting $e_{\nu}(0)=0$ and
\begin{equation}\label{eq:e_from_z}
e_{\nu}(a):=
    a \int_{a_1}^a \frac{z'_{\nu}(f'_{\nu}(\alpha))}{\alpha^2}\, d\alpha \quad \textup{for} \; a\in \dom(f_{\nu})\cap (0,\infty), 
\end{equation}
then $e_{\nu}$ is convex and satisfies the coupling relation \eqref{eq:abc_equation} with $z_{\nu}$.
Furthermore, $e_{\nu}$ is unique up to a linear factor.
\end{lemma}
\begin{remark}
 The freedom of choice in the parameter $a_1$ is why $e_{\nu}$ is only unique up to a linear factor.  Note however, that linear factors do not affect the energy evolution equation \eqref{eq:formal_energy_equation}. Indeed, the linear term can be separated out, and one can then check that it must vanish, since $\rhonu$ itself solves the original equation \eqref{eq: rho}. 
\end{remark}

\smallskip

For the sake of exposition, the proofs of Lemma~\ref{lem:z_from_e} and \ref{lem:e_from_z} are presented in Appendix~\ref{app:1}.

\medskip

 Now we are ready to rigorously state and prove the energy evolution equation, which we will establish in the weak form \eqref{eq:eee}.
\begin{prop}[The Energy Evolution Equation]\label{prop:eee}
Suppose that $e_{\nu}:\dom(f_{\nu})\to\RR$ and $z_{\nu}:[0,\infty)\to\RR$ are convex functions satisfying the coupling relation \eqref{eq:abc_equation} with $e_{\nu}(0)=0$, $z_{\nu}\in W^{2,1}_{\loc}([0,\infty))$, and $z_{\nu}(0)=z_{\nu}'(0)=0$.  
 Given any smooth test function $\psi\in L^{\infty}_c([0,\infty);L^1(\RR^d))$, $e_{\nu}(\rhonu)$ satisfies the weak evolution equation:
\begin{equation}\label{eq:eee}
\begin{split} 
 &\int_0^{\infty}\!\!\! \int_{\RR^d} {\psi \big(z_{\nu}'(\pnu
)-z_{\nu}'(\wnu)\big)\frac{\pnu-\wnu}{\nu}}+\psi z_{\nu}''(\wnu)|\nabla \wnu|^2  \\[0.3em]
&\qquad\qquad+\int_0^{\infty}\!\!\! \int_{\RR^d} e_{\nu}(\rhonu)(\nabla \wnu\cdot \nabla \psi -\partial_t \psi)-z_{\nu}(\wnu)\Delta \psi \\[0.5em]
= &\int_{\RR^d} \psi e_{\nu}(\rhonu^{\mathrm{in}})+\int_0^{\infty}\!\!\! \int_{\RR^d}  \psi \big(e_{\nu}(\rhonu)+z'_{\nu}(\pnu)\big)G(\pnu).
\end{split}
\end{equation}
\end{prop}

\begin{remark}
Note that $z_{\nu}\in W^{2,1}_{\loc}([0,\infty))$ is sufficient for the  the dissipation term \[
\int_0^{\infty}\!\!\! \int_{\RR^d}\psi z_{\nu}''(\wnu)|\nabla \wnu|^2 \dx x \dx t
\]
to be well-defined. The measure $|\nabla w_{\nu}(t,x)|^2\, \dx x \dx t$ gives no mass to sets of the form $\{(t,x): \wnu(t,x)\in B\}$ when $B\subset \RR$ has Lebesgue measure zero.  Since $z''_{\nu}(b)$ is defined for pointwise almost every $b\in \RR$, the product  $z_{\nu}''(\wnu)|\nabla \wnu|^2$ is well defined pointwise almost everywhere.  To see that $z_{\nu}''(\wnu)|\nabla \wnu|^2\in L^1_{\loc}([0,\infty)\times\RR^d)$, we can apply the coarea formula to write
 \[
\int_0^{\infty}\!\!\! \int_{\RR^d}\psi z_{\nu}''(\wnu)|\nabla \wnu|^2=\int_0^{\infty} z''_{\nu}(\beta) \eta(\beta) d\beta
 \]
 where $\eta:[0,\infty)\to\RR$ is given by 
\begin{equation}\label{eq:g_def}
\eta(\beta):= \iint_{\{(t,x):w(t,x)=\beta\}} \psi(t,x)\frac{ |\nabla \wnu(t,x)|^2 }{\sqrt{|\nabla \wnu(t,x)|^2+|\partial_t \wnu(t,x)|^2}} dH^{d-1}(x)\, dt.
 \end{equation}
  In the subsequent proof,  we will show that $\eta\in L^{\infty}_c([0,\infty))$, which is what is needed to conclude that $z_{\nu}''(\wnu)|\nabla \wnu|^2\in L^1_{\loc}([0,\infty)\times\RR^d)$.
\end{remark}


\medskip

\begin{proof} 

Since $\psi$ is compactly supported in time, there exists some $T>0$ such that $\spt(\psi)\subset [0,T]\times\RR^d$.
Since $\Delta w_{\nu}\in L^{\infty}([0,T]\times\RR^d)$, the Lagrangian flow
\[
\partial_t X_{\nu}(t,x)=-\nabla \wnu(t,X_{\nu}(t,x))
\]
is well-posed, invertible and $|\det(DX_{\nu})|+|\det(DX_{\nu})|^{-1}\in L^{\infty}([0,T]\times\RR^d)$ for all $T>0$.  Let us define

\[
\trhonu:=\rhonu\circ X_{\nu}, \quad  \twnu=\wnu\circ X_{\nu}, \quad \tpnu:=\pnu\circ X_{\nu}, \quad \tilde{\psi}_{\nu}=X^{-1}_{\nu\,\#}\psi.
\]
It will be useful to note that $\tilde{\psi}_{\nu}=\psi\circ X_{\nu} \det(DX_{\nu})$ and $X_{\nu\,\#}\tilde{\psi}_{\nu}=\psi$.
From the properties of $X_{\nu}$ and the $\rho_{\nu}$-equation, it follows that $\tilde{\rho}_\nu$ is a weak solution to
\begin{equation}\label{eq:rho_exponential}
\partial_t \trhonu=\trhonu\big((\Delta \wnu)\circ X_{\nu}+G(\tpnu)\big).
\end{equation}
Choosing some $\tau>0$ and averaging both sides of (\ref{eq:rho_exponential}) on the time interval $[t,t+\tau]$, we can deduce that
\begin{equation}\label{eq:time_average}
\frac{\sigma_{\tau} \trhonu-\trhonu}{\tau}=\frac{1}{\tau}\int^{t+\tau}_t\trhonu\big((\Delta \wnu)\circ X_{\nu}+G(\tpnu)\big),
\end{equation}
pointwise almost everywhere on $[0,\infty)\times\RR^d$,
where $\sigma_{\tau}$ is the forward time shift operator.

For each $\delta>0$ choose $Z_{\nu, \delta}\in C^{\infty}(\RR)$ such that $Z_{\nu,\delta}$ is nonnegative, compactly supported in $(0,\infty)$, and $Z_{\nu,\delta}$ converges to $z_{\nu}''$ in $L^1_{\loc}([0,\infty))$ as $\delta\to 0$.  We then define convex functions $z_{\nu,\delta}, e_{\nu,\delta}$ by taking
\[
z_{\nu,\delta}'(b):=\int_0^{b} Z_{\nu,\delta}(\beta)\dx\beta, \qquad
e_{\nu,\delta}(a):=a\int_{a_0/2}^a \frac{z'_{\nu,\delta}(f'_{\nu}(\alpha))}{\alpha^2}\, \dx\alpha,
\]
where $a_0>0$ is the value guaranteed by our assumptions on the internal energies (see Assumption \ref{as: energies} part 3).  
Applying  Lemma \ref{lem:e_from_z},  it follows that $e_{\nu,\delta}$ is convex, $e_{\nu}(0)=0$, and $e_{\nu,\delta}, z_{\nu,\delta}$ satisfy the coupling relation \eqref{eq:abc_equation}.  Furthermore, we have 
\[
\lambda (\rhonu-\min(a_0/2,\rhonu))\leq e_{\nu,\delta}(\rhonu)\leq \rhonu z'_{\nu,\delta}(\pnu)\left(\frac{2}{a_0}-\frac{1}{\max(\rhonu, a_0)}\right),
\]
where $\lambda=-\int_{0}^{a_0/2} \frac{z'_{\nu,\delta}(f'_{\nu}(\alpha))}{\alpha^2}\, \dx\alpha.$ Hence, $e_{\nu,\delta}(\rhonu)\in L^{\infty}([0,T]\times\RR^d)$ independently of $\delta$.

Define a new variable $\zeta_{\nu,\delta}:[0,\infty)\times\RR^d\to \RR$ by setting
\[
\zeta_{\nu,\delta}:=\begin{dcases}
\frac{e_{\nu,\delta}(\rhonu)+z'_{\nu,\delta}(\pnu)}{\rhonu} & \textup{if} \quad \rhonu\neq 0,\\
   -\int_{0}^{a_0/2} \frac{z'_{\nu,\delta}(f'_{\nu}(\alpha))}{\alpha^2}\, \dx\alpha & \textup{otherwise}.\\ 
\end{dcases}
\]
Since $z'_{\nu,\delta}(f'_{\nu}(\alpha))=0$ for all $\alpha$ sufficiently close to zero, $\rhonu, \pnu\in L^{\infty}([0,\infty)\times \RR^d)$ implies that $\zeta_{\nu,\delta}\in L^{\infty}([0,\infty)\times \RR^d)$.  Furthermore, combining the coupling relation \eqref{eq:abc_equation} with the definition of $\zeta_{\nu,\delta}$, it follows that $\zeta_{\nu,\delta}\in \partial e_{\nu,\delta}(\rhonu)$ almost everywhere.  

\smallskip

Define $\tilde{\zeta}_{\nu,\delta}:=\zeta_{\nu,\delta}\circ X_{\nu}$.
Since $\tilde{\zeta}_{\nu,\delta}\in L^{\infty}_{\loc}([0,\infty)\times \RR^d)$,  it is valid to integrate both sides of \eqref{eq:time_average} by $\tilde{\psi} \tilde{\zeta}_{\nu,\delta}$ on $[0,\infty)\times\RR^d$ to obtain the equality 
\begin{equation}\label{eq:transport_dissipation}
\int_0^{\infty}\!\!\! \int_{\RR^d}\frac{\sigma_{\tau} \trhonu-\trhonu}{\tau}\tilde{\zeta}_{\nu,\delta}\tilde{\psi}_{\nu}=\int_0^{\infty}\!\!\! \int_{\RR^d}\tilde{\zeta}_{\nu,\delta}\tilde{\psi}_{\nu}\frac{1}{\tau}\int^{t+\tau}_t\trhonu\big((\Delta \wnu)\circ X_{\nu}+G(\tpnu)\big).
\end{equation}
Using the convexity of $e_{\nu,\delta}$ and the subdifferential condition $\zeta_{\nu,\delta}\in \partial e_{\nu,\delta}(\rhonu)$, we have the inequality
\begin{equation}\label{eq:youngs_time_trick}
\int_0^{\infty}\!\!\! \int_{\RR^d}\frac{\sigma_{\tau} \trhonu-\trhonu}{\tau}\tilde{\zeta}_{\nu,\delta}\tilde{\psi}_{\nu}\leq \int_0^{\infty}\!\!\! \int_{\RR^d}\frac{\sigma_{\tau} e_{\nu,\delta}(\trhonu)-e_{\nu,\delta}(\trhonu)}{\tau}\tilde{\psi}_{\nu}.
\end{equation}
If we then change variables to move the time finite difference operator onto $\tilde{\psi}_{\nu}$, the right-hand-side of (\ref{eq:youngs_time_trick}) becomes
\[
-\frac{1}{\tau}\int_0^{\tau}\!\! \int_{\RR^d} e_{\nu,\delta}(\trhonu)\tilde{\psi}_{\nu}+ \int_\tau^{\infty}\!\!\! \int_{\RR^d}  e_{\nu,\delta}(\trhonu)\frac{\sigma_{-\tau}\tilde{\psi}_{\nu}-\tilde{\psi}_{\nu}}{\tau}.
\]
Combining this with \eqref{eq:transport_dissipation} and then sending $\tau\to 0$, we obtain
 \[
\int_0^{\infty}\!\!\! \int_{\RR^d}\tilde{\zeta}_{\nu,\delta}\tilde{\psi}_{\nu}\trhonu\big((\Delta \wnu)\circ X_{\nu}+G(\tpnu)\big)\leq -\int_{\RR^d}  e_{\nu,\delta}(\trhonu(0))\tilde{\psi}_{\nu}(0)+ \int_0^{\infty}\!\!\! \int_{\RR^d}  e_{\nu,\delta}(\trhonu)\partial_t \tilde{\psi}_{\nu}.
\]
Note that if we had integrated \eqref{eq:time_average} against $\sigma_{\tau}\big( \tilde{\zeta}_{\nu,\delta}\tilde{\psi}_{\nu}\big)$ instead of $\tilde{\zeta}_{\nu,\delta}\tilde{\psi}_{\nu}$, we would have obtained the opposite inequality.  Hence, we in fact have 
\[
\int_0^{\infty}\!\!\! \int_{\RR^d}\tilde{\zeta}_{\nu,\delta}\tilde{\psi}_{\nu}\trhonu \big((\Delta \wnu)\circ X_{\nu}+G(\tpnu)\big)= -\int_{\RR^d}  e_{\nu,\delta}(\trhonu(0))\tilde{\psi}_{\nu}(0)+ \int_0^{\infty}\!\!\! \int_{\RR^d}  e_{\nu,\delta}(\trhonu)\partial_t \tilde{\psi}_{\nu}.
\]
Pushing forward by $X_{\nu}$ on both sides and rearranging, the above is equivalent to \begin{equation}\label{eq:transport_dissipation_1}
\int_0^{\infty}\!\!\! \int_{\RR^d}\zeta_{\nu,\delta}\psi\rhonu \big(\Delta \wnu+G(\pnu)\big)+\int_{\RR^d}  e_{\nu,\delta}(\rhonu^{\mathrm{in}})\psi =\int_0^{\infty}\!\!\! \int_{\RR^d}  e_{\nu,\delta}(\rhonu)(\partial_t\tilde{\psi}_{\nu})\circ X^{-1}_{\nu}\det(DX_{\nu}^{-1}).
\end{equation}
Using the determinant formula for $\tilde{\psi}_{\nu}=X_{\nu\,\#}^{-1}\psi$, we can compute
\[
\partial_t \tilde{\psi}_\nu\circ X_{\nu}^{-1}=(\partial_t \psi-\nabla \psi\cdot \nabla \wnu-\psi \Delta \wnu)\det(DX_{\nu})\circ X_{\nu}^{-1}.
\]
Hence, using the fact that $\det(DX_{\nu})\circ X_{\nu}^{-1}=\det(DX^{-1}_{\nu})^{-1}$ it follows that
\[
\partial_t\tilde{\psi}_{\nu}\circ X^{-1}_{\nu}\det(DX_{\nu}^{-1})=\partial_t \psi-\nabla \psi\cdot \nabla \wnu-\psi\Delta \wnu.
\]
Applying this to \eqref{eq:transport_dissipation_1}, we get
\begin{equation*}\label{eq:almost_done_transport_dissipation}
    \int_0^{\infty}\!\!\! \int_{\RR^d} \psi\rhonu \zeta_{\nu,\delta}\big(\Delta \wnu+G(\pnu)\big)+\int_{\RR^d}  e_{\nu,\delta}(\rhonu^{\mathrm{in}})\psi= \int_0^{\infty}\!\!\! \int_{\RR^d} e_{\nu,\delta}(\rhonu)\big(\nabla \cdot (\psi \nabla \wnu)-\partial_t \psi\big).
\end{equation*}
Using the crucial identity $\rhonu \zeta_{\nu,\delta}=e_{\nu,\delta}(\rhonu)+z_{\nu,\delta}'(\pnu)$, we see that the previous line is equivalent to
\begin{equation*}\label{eq:almost_done_transport_dissipation_1}
    \int_0^{\infty}\!\!\! \int_{\RR^d}e_{\nu,\delta}(\rhonu)(\nabla \psi\cdot \nabla \wnu-\partial_t \psi) -\psi\big(e_{\nu,\delta}(\rhonu) +z_{\nu,\delta}'(\pnu)\big)G(\pnu)-\psi z_{\nu,\delta}'(\pnu)\Delta \wnu=\int_{\RR^d}  e_{\nu,\delta}(\rhonu^{\mathrm{in}})\psi.
\end{equation*}
Adding and subtracting $\psi z_{\nu,\delta}'(\wnu)\Delta \wnu$ and using $\nu\Delta\wnu=\wnu-\pnu$, we obtain 
\begin{align*}\label{eq:one_sided}
\int_0^{\infty}\!\!\! \int_{\RR^d}e_{\nu,\delta}(\rhonu)(\nabla \psi\cdot \nabla \wnu-\partial_t \psi)&+\psi \big(z_{\nu,\delta}'(\pnu)-z_{\nu,\delta}'(\wnu))\frac{\pnu-\wnu}{\nu}-\psi z_{\nu,\delta}'(\wnu)\Delta \wnu\\[0.3em]
&=\int_0^{\infty}\!\!\! \int_{\RR^d}\psi\big(e_{\nu,\delta}(\rhonu)+z'_{\nu,\delta}(\pnu)\big)G(\pnu)+\int_{\RR^d}  e_{\nu,\delta}(\rhonu^{\mathrm{in}})\psi.
\end{align*}
Rearranging and using $\Delta z_{\nu,\delta}(\wnu)=Z_{\nu,\delta}(\wnu)|\nabla \wnu|^2+z_{\nu,\delta}'(\wnu)\Delta \wnu$, this becomes
\begin{equation}\label{eq:edi_delta}
\begin{split}
 &\int_0^{\infty}\!\!\!\int_{\RR^d} {\psi \big(z_{\nu,\delta}'(\pnu
)-z_{\nu,\delta}'(\wnu)\big)\frac{\pnu-\wnu}{\nu}}+\psi Z_{\nu,\delta}(\wnu)|\nabla \wnu|^2  \\[0.3em]
&\qquad\qquad+\int_0^{\infty}\!\!\!\int_{\RR^d}  e_{\nu,\delta}(\rhonu)(\nabla \psi\cdot \nabla \wnu -\partial_t \psi)-z_{\nu,\delta}(\wnu)\Delta \psi\\[0.5em]
&= \int_{\RR^d}e_{\nu,\delta}(\rhonu^{\mathrm{in}}) \psi +\int_0^{\infty}\!\!\!\int_{\RR^d}  \psi \big(e_{\nu,\delta}(\rhonu)+z'_{\nu,\delta}(\pnu)\big)G(\pnu),
\end{split}
\end{equation} 
which is \eqref{eq:eee} except that $e_{\nu}$ and $z_{\nu}$ are replaced by $e_{\nu,\delta}$ and $z_{\nu,\delta}$ and $z''_{\nu}$ is replaced by $Z_{\nu,\delta}$.

Taking $\delta\to 0$, it is clear that $z_{\nu,\delta}'$ converges uniformly on compact subsets to $z_{\nu}'$ and $e_{\nu,\delta}$ converges uniformly on compact subsets to a convex function $\bar{e}_{\nu}$.  Since the coupling relation \eqref{eq:abc_equation} holds for all $\delta>0$, it follows that $z_{\nu}'$ and $\bar{e}_{\nu}$ must also satisfy the coupling relation.  By Lemma \ref{lem:e_from_z}, this implies that $\bar{e}_{\nu}(a)=e_{\nu}(a)+\lambda a$ for some fixed $\lambda\in  \RR$.   Passing to the limit $\delta\to 0$ in \eqref{eq:edi_delta}, and using the fact that the equation is invariant under linear shifts of $e_{\nu}$, we can conclude that
\begin{equation}\label{eq:edi_delta_limit}
\begin{split}
 &\lim_{\delta \to 0}\int_0^{\infty}\!\!\!\int_{\RR^d} \psi Z_{\nu,\delta}(\wnu)|\nabla \wnu|^2   \\[0.1em]
&+\int_0^{\infty}\!\!\!\int_{\RR^d}  e_{\nu}(\rhonu)(\nabla \wnu\cdot \nabla \psi -\partial_t \psi)-z_{\nu}(\wnu)\Delta \psi+{\psi \big(z_{\nu}'(\pnu
)-z_{\nu}'(\wnu)\big)\frac{\pnu-\wnu}{\nu}} \\[0.5em]
&= \int_{\RR^d} \psi e_{\nu}(\rhonu^{\mathrm{in}})+\int_0^{\infty}\!\!\!\int_{\RR^d}  \psi \big(e_{\nu}(\rhonu)+z'_{\nu}(\pnu)\big)G(\pnu).
\end{split}
\end{equation} 
By the coarea formula, 
\[
\int_0^{\infty}\!\!\!\int_{\RR^d} \psi Z_{\nu,\delta}(\wnu)|\nabla \wnu|^2=\int_0^{\infty} Z_{\nu,\delta}(\beta)\eta(\beta)\,d\beta,
\]
where $\eta(\beta)\geq 0$ is given by \eqref{eq:g_def}.  
If $\beta$ is a regular value of $w_{\nu}$, then $\{(t,x):w(t,x)=\beta\}=\partial \{(t,x):w(t,x)\leq \beta\}$, and we can use integration by parts to write
\begin{equation}\label{eq:g_by_parts}
\eta(\beta)\leq \int_0^{\infty}\!\!\! \int_{\RR^d} \chi_{\beta}(w_{\nu})\nabla \cdot (\psi\nabla \wnu),
\end{equation}  
where $\chi_{\beta}$ is the characteristic function of $[0,\beta]$.    Since singular values of $w_{\nu}$ have Lebesgue measure zero, we can combine equation (\ref{eq:g_by_parts}) with the control $\wnu, \nabla \wnu, \Delta w_{\nu}\in L^{\infty}_{\loc}([0,\infty);L^{\infty}(\RR^d))$ to conclude that $\eta\in L_c^{\infty}([0,\infty))$. 
  
Finally, since $Z_{\nu,\delta}$ converges in $L^1_{\loc}([0,\infty))$ to $z_{\nu}''$, and $\eta\in L^{\infty}_c([0,\infty))$, it follows that
\[
\lim_{\delta\to 0}\int_0^{\infty} Z_{\nu,\delta}(\beta)\eta(\beta)\,d\beta=\int_0^{\infty} z_{\nu}''(\beta) \eta(\beta)=\int_0^{\infty}\!\!\! \int_{\RR^d} \psi z_{\nu}''(\wnu)|\nabla \wnu|^2,
\]
and we are done.
\end{proof}

\medskip

The following Corollary establishing dissipation of the total energy $\int_{\RR^d} e_{\nu}(\rhonu)$, is one of the most useful consequences of the energy evolution equation.
\begin{cor}[Energy dissipation]\label{cor:dissipation}
    Suppose that $e_{\nu}:\dom(f_{\nu})\to\RR$ and $z_{\nu}:[0,\infty)\to\RR$  satisfy the assumptions of Proposition \ref{prop:eee}.  If $\liminf_{a\to 0} \frac{e_{\nu}(a)}{a}\neq -\infty$, then for any time $T\geq 0$ and any $\eta\in W^{1,\infty}([0,T])$,
\begin{equation}\label{eq:edi_T}
\begin{split}
&\int_{\RR^d}  \eta(T) e_{\nu}(\rhonu(T))+\int_0^T\!\!\!\int_{\RR^d} { \eta\big(z_{\nu}'(\pnu
)-z_{\nu}'(\wnu)\big)\frac{\pnu-\wnu}{\nu}}+ \eta z_{\nu}''(\wnu)|\nabla \wnu|^2-e_{\nu}(\rhonu)\partial_t \eta  \\[0.1em]
&=  \int_{\RR^d}  \eta(0) e_{\nu}(\rhonu^{\mathrm{in}})+\int_0^T\!\!\!\int_{\RR^d} \eta G(\pnu)\big(e_{\nu}(\rhonu)+z_{\nu}'(\pnu)\big).
\end{split}
\end{equation}
\end{cor}
\begin{proof}
Let $\vp\in C^{\infty}_c(\RR)$ and $\zeta\in C^{\infty}_c([0,\infty))$ be nonnegative functions such that $\vp(r)=1$ for $|r|\leq 1$, $\vp(r)=0$ for $|r|\geq 2$. For each $k\in \ZZ_+$, define $\psi_k:[0,\infty)\times\RR^d\to \RR$ by setting $\psi_k(t,x):=\vp(|x|/k)\zeta(t)$. It then follows that $\psi_k$ is an increasing sequence of nonnegative functions, such that $\psi_k\in C_c^{\infty}([0,\infty)\times\RR^d)$ for all $k$.  Thus, \eqref{eq:eee} must hold for each $\psi_k$. It is clear that as $k\to\infty$, we have the pointwise almost everywhere convergence of $\psi_k, \partial_t \psi_k, \nabla \psi_k, \Delta \psi_k$ to $\zeta, \partial_t \zeta$, zero, and zero respectively. However,  in order to pass to the limit $k\to\infty$ in \eqref{eq:eee}, we will need to show that various quantities belong to $L^1$. 

\medskip

First, the control $\liminf_{a\to 0} \frac{e_{\nu}(a)}{a}\neq -\infty$ together with the $L^1$ bounds on $\rhonu$ and $L^{\infty}$ bounds on $e_{\nu}$ implies that $e_{\nu}(\rhonu)$ belongs to $L^1([0,T]\times\RR^d)$. Next, since $z_{\nu}'\in W^{1,1}_{\loc}([0,\infty))$, the $L^{\infty}$ bounds on $\pnu$ imply that $z_{\nu}(\pnu), z_{\nu}'(\pnu) \in L^{\infty}([0,T]\times\RR^d)$.  If we choose a value $a\in \dom(f_{\nu})\cap (0,\infty)$, such that $e_{\nu}$ is differentiable at $a$, then on the set $\{\rhonu\leq a\}$, we have $0\leq z'(\pnu)\leq |e_{\nu}(\rhonu)|+\rhonu e'_{\nu}(a)$.  Thus, it follows that $z'_{\nu}(\pnu), z_{\nu}(\pnu)\in L^1([0,T]\times\RR^d)$.  Finally, the convexity of $z_{\nu}$ together with the elliptic equation $\wnu-\nu\Delta \wnu=\pnu$ implies that $\norm{z_{\nu}(\wnu)}_{L^1([0,T]\times \RR^d)}\leq \norm{z_{\nu}(\pnu)}_{L^1([0,T]\times \RR^d)}$.

\medskip

Now we can pass to the limit $k\to\infty$ in all of the terms in   \eqref{eq:eee}, either by combining the $L^1$ bounds proven above with the dominated convergence theorem or by applying the monotone convergence theorem to the nonnegative terms.   Hence, we have
\begin{equation}
\begin{split}
&\int_0^{\infty}\!\!\! \int_{\RR^d}{\zeta \big(z_{\nu}'(\pnu
)-z_{\nu}'(\wnu)\big)\frac{\pnu-\wnu}{\nu}}+\zeta z_{\nu}''(\wnu)|\nabla \wnu|^2 -e_{\nu}(\rhonu)\partial_t \zeta \\[0.5em]
&= \int_{\RR^d} \zeta e_{\nu}(\rhonu^{\mathrm{in}})+\int_0^{\infty}\!\!\! \int_{\RR^d}  \zeta \big(e_{\nu}(\rhonu)+z'_{\nu}(\pnu)\big)G(\pnu).
\end{split}
\end{equation}
By approximation, it is clear that the above holds for any $\eta\in W^{1,\infty}_c([0,\infty))$.
Finally, if we choose some $\tau>0$ and define $\zeta(t):=\eta(t)\omega_{\tau}(t)$ where 
\[
\omega_{\tau}(t):=\begin{cases}
     1 & \textup{if}\; t\leq T-\tau,\\
     1-\frac{t+\tau-T}{\tau} & \textup{if} \; t\in [T-\tau,T],\\
     0 & \textup{otherwise,}
\end{cases}
\]
then we have
\begin{equation*}
\begin{split}
&\int_0^\infty\!\!\!\int_{\RR^d}{\eta\omega_{\tau} \big(z_{\nu}'(\pnu
)-z_{\nu}'(\wnu)\big)\frac{\pnu-\wnu}{\nu}}+\eta \omega_{\tau} z_{\nu}''(\wnu)|\nabla \wnu|^2-e_{\nu}(\rhonu)\omega_{\tau}\partial_t \eta +\frac{1}{\tau}\int_{T-\tau}^T\!\!\int_{\RR^d}\eta e_{\nu}(\rhonu) \\[0.5em]
&= \int_{\RR^d} \eta(0) e_{\nu}(\rhonu^{\mathrm{in}})+\int_0^\infty\!\!\!\int_{\RR^d} \eta\omega_{\tau} \big(e_{\nu}(\rhonu)+z'_{\nu}(\pnu)\big)G(\pnu).
\end{split}
\end{equation*}
Sending $\tau\to 0$, we get the desired result. 
\end{proof}

\subsection{Applications of energy dissipation}

We now illustrate applications of Proposition \ref{prop:eee} and Corollary \ref{cor:dissipation} where we make specific choices for $e_{\nu}$ and $z_{\nu}$. 
First, we recall some properties of Darcy's law and its relation to gradient flows.
When $G=0$, it is by now well-known (see Otto \cite{otto_pme})  that \eqref{eq:limit} is the gradient flow of the energy
\[
F_{0}(\rho):=\int_{\RR^d} f_0(\rho) \dx x,
\]
with respect to the $W_2$ distance. It is another classical fact that \eqref{eq:limit} is also a gradient flow (again with $G=0$) in the $H^{-1}$ distance, of the energy
\[
H_0(\rho):=\int_{\RR^d} h_0(\rho) \dx x, \qquad \hbox{ where } \,\, h_0(a)=af_0(a)-2\int_0^a f_0(\alpha) d\alpha.
\]  
In both cases, the gradient flow structure can be used to establish a dissipation relation for solutions of \eqref{eq:limit}, between the rate of change in the energy and the "kinetic work" generated by the evolution of the density in the given metric spaces, 
 with a correction term due to the presence of $G$ (see for instance \cite{jkt21}). The next two Corollaries show that one has analogues of these dissipation relations for Brinkman's flow \eqref{eq: rho} when $\nu>0$. In Section \ref{sec: strong comp}, the $H^{-1}$ dissipation relation, Corollary \ref{cor:h-1}, will play a key role in our convergence argument as we explain in the remarks below.

\begin{cor}[Internal energy dissipation]\label{cor:ede}
For any $T>0$ and $\eta\in W^{1,\infty}([0,T])$,
\begin{equation}
\begin{split}
    \label{eq:edi f}
\int_{\RR^d} \eta(T) f_{\nu}(\rhonu(T)) + &\int_0^T \!\!\!\int_{\RR^d} \eta f^{*\, \prime}_{\nu}(\wnu) |\nabla \wnu|^2+\eta (f^*_{\nu}(\pnu)-f^*_{\nu}(\wnu))\frac{\pnu-\wnu}{\nu}-f_{\nu}(\rhonu)\partial_t \eta\\
&\qquad
= \int_0^T\!\!\! \int_{\RR^d} \eta\rhonu\pnu G(p_\nu)+\int_{\RR^d} \eta(0) f_{\nu}(\rhonu^{\mathrm{in}}).
\end{split}
\end{equation}
\end{cor}

\begin{proof}
   Choose $e_{\nu}=f_{\nu}$ and $z'_{\nu}=f^*_{\nu}$.  The compatibility condition  \eqref{eq:abc_equation} follows immediately from the definition of the Legendre transform. The above equation is then precisely \eqref{eq:edi_T} for $f_{\nu}$ and $f_{\nu}^*$, except we have used the pressure-density duality relation to rewrite $\big(f_{\nu}(\rhonu)+f^*_{\nu}(\pnu)\big) G(p_\nu)=\rhonu\pnu G(p_\nu)$.
\end{proof}

\begin{remark}\label{rmk:DEI}
    In the setting of \eqref{eq:limit}, the kinetic work in the $W_2$ distance is the space-time integral of $\rho |\nabla p|^2$. This is replaced by the integral of $ f^{*\, \prime}_{\nu}(\wnu)|\nabla \wnu|^2$ for Brinkman's law. Although we do not use this relation in our convergence arguments, we believe it provides an important philosophical link between the behavior of Brinkman's law and the behavior of Darcy's law.
\end{remark}

\begin{cor}[$H^{-1}$ energy dissipation]\label{cor:h-1}
Let $h_{\nu}:\dom(f_{\nu})\to \RR$  be given by 
\[
h_{\nu}(a):=af_{\nu}(a)-2\int_0^a f_{\nu}(\alpha)\, d\alpha.
\]
For any $T>0$ and any $\eta\in W^{1,\infty}([0,T])$,
\begin{equation}
\begin{split}
    \label{eq:edi h}
\int_{\RR^d} \eta(T) h_{\nu}(\rhonu(T)) + &\int_0^T \!\!\!\int_{\RR^d}  \eta |\nabla f_\nu^*(\wnu)|^2+\eta\big(h^*_{\nu}(f^*_{\nu}(\pnu))-h^*_{\nu}(f^*_{\nu}(\wnu))\big)\frac{\pnu-\wnu}{\nu}-h_{\nu}(\rhonu)\partial_t \eta\\
&\qquad
= \int_0^T\!\!\! \int_{\RR^d} \eta\rhonu f^*_{\nu}(\pnu)G(p_\nu)+\int_{\RR^d} \eta h_{\nu}(\rhonu^{\mathrm{in}}).
\end{split}
\end{equation}
\end{cor}
\begin{proof}
   Set $e_{\nu}(a)=h_{\nu}(a)$ and $z'_{\nu}(b)=h_{\nu}^*(f_{\nu}^*(b))$. It is then clear that $e_{\nu}(0)=0$ and $z'_{\nu}(0)=0$. If $b\in \partial f_{\nu}(a)$, then from the definition of $h_\nu$, we have $c:=ab-f_{\nu}(a)\in \partial h_{\nu}(a)$.  The relation $b\in \partial f_{\nu}(a)$ also implies that $c=ab-f_{\nu}(a)=f_{\nu}^*(b)$. We then see that 
    \[
    ac-h_{\nu}(a)=af_{\nu}^*(b)-h_{\nu}(a)=h^*_{\nu}(f_{\nu}^*(b)),
    \]
   hence, $h_{\nu}$ and $h_{\nu}^*\circ f_{\nu}^*$ satisfy the coupling relation \eqref{eq:abc_equation}. It also follows from the coupling relation that $\rhonu f^*_{\nu}(\pnu)=h_{\nu}(\rhonu)+ h^*_{\nu}(f^*_{\nu}(\pnu))$.  The result is then given by  \eqref{eq:edi_T} for the choices $h_{\nu}$ and $h_{\nu}^*\circ f_{\nu}^*$, where we have used the identity $\rhonu f^*_{\nu}(\pnu)=h_{\nu}(\rhonu)+ h^*_{\nu}(f^*_{\nu}(\pnu))$ to simplify the factor in front of $G(\pnu)$.
\end{proof}

\begin{remark}\label{rmk:DEI_h1}
    In the setting of \eqref{eq:limit},  the kinetic work in the $H^{-1}$ distance  is the space-time integral of $|\nabla f^*_0(p)|^2$. This is replaced by the integral of $ |\nabla f^*_{\nu}(\wnu)|^2$. In Section~\ref{sec: strong comp}, we will show that this quantity converges as $\nu\to 0$, which yields the strong convergence of $\nabla w_{\nu}$ to  $\nabla p$ in $L^2$.
\end{remark} 

\begin{cor}[Control on derivatives of $\wnu$]\label{cor:w_control}
Define $e_{\nu}:\dom(f_{\nu})\to\RR$ by taking
\[
e_{\nu}(a):=a\int_{a_0/2}^a \frac{f'_{\nu}(\alpha)}{\alpha^2}\, d\alpha,
\]
where $a_0>0$ is the value identified in Assumption \ref{as: energies} on the internal energies.

\medskip
\noindent
If $\liminf_{a\to 0} \frac{e_{\nu}(a)}{a}\neq -\infty$, then for any time $T\geq 0$, we have
\begin{equation}\label{e.EDI z''=1}
\begin{split}
\int_{\RR^d} e_\nu(\rho_{\nu}(T))+\int_0^T \!\!\!\int_{\RR^d} |\nabla \wnu|^2+\nu|\Delta \wnu|^2= \int_0^T \!\!\!\int_{\RR^d} \big(e_{\nu}(\rhonu)+\pnu\big) G(\pnu)+ \int_{\RR^d} e_\nu(\rho_{\nu}^{\mathrm{in}}).
\end{split}
\end{equation}

\smallskip
\noindent
Otherwise, given any $T\geq 0$ and any test function $\psi:[0,\infty)\times\RR^d\to [0,\infty)$ with compact support in time and
\begin{equation}\label{eq:n_psi}
N(\psi):=\norm{\psi}_{L^1(\{0\}\times \RR^d)}+\norm{\psi}_{L^1([0,\infty)\times \RR^d)}+\norm{\partial_t \psi}_{L^1([0,\infty)\times \RR^d)}+\left\|\frac{|\nabla \psi|^2}{\psi}\right\|_{L^1([0,\infty)\times\RR^d)}<\infty,
\end{equation}
there exists a constant $C>0$ such that 
\begin{equation}\label{eq:w_local_bound}
\int_0^{\infty}\!\!\! \int_{\RR^d} |\nabla \wnu|^2\psi+\nu|\Delta \wnu|^2\psi\leq C N(\psi)\big(\norm{\rhonu+\pnu+\rhonu^2\pnu^2}_{L^{\infty}([0,\infty)\times\RR^d)}+1\big).
\end{equation}
\end{cor} 
\begin{proof}
   If we take $z_{\nu}'(b)=b$, then by Lemma \ref{lem:e_from_z}, it follows that $e_{\nu} $ and $z'_{\nu}$ satisfy the coupling relation.  The first result then follows from Corollary \ref{cor:dissipation}  with $\eta=1$.

Otherwise, using $\psi$, the energy evolution equation, Proposition \ref{prop:eee}, gives us
\begin{equation}
\begin{split} 
 &\int_0^{\infty}\!\!\! \int_{\RR^d}  \nu^{-1}\big(\pnu
-\wnu\big)^2\psi+ |\nabla \wnu|^2\psi+\int_0^{\infty}\!\!\! \int_{\RR^d} e_{\nu}(\rhonu)(\nabla \wnu\cdot \nabla \psi -\partial_t \psi)-\wnu\Delta \psi \\[0.5em]
&\qquad= \int_{\RR^d} \psi e_{\nu}(\rhonu^{\mathrm{in}})+\int_0^{\infty}\!\!\! \int_{\RR^d}  \psi \big(e_{\nu}(\rhonu)+\pnu\big)G(\pnu).
\end{split}
\end{equation}
Integrating by parts the term $-\wnu\Delta\psi$ and using Young's inequality on both terms involving $\nabla \wnu\cdot \nabla \psi$, it follows that
\begin{equation}
\begin{split} 
 &\int_0^{\infty}\!\!\! \int_{\RR^d}  \nu^{-1}\big(\pnu
-\wnu\big)^2\psi+ \frac{1}{2}|\nabla \wnu|^2\psi-e_{\nu}(\rhonu)\partial_t \psi \\[0.5em]
&\qquad\leq \int_{ \RR^d} \psi e_{\nu}(\rhonu^{\mathrm{in}})+\int_0^{\infty}\!\!\! \int_{\RR^d}  \psi \big(e_{\nu}(\rhonu)+\pnu\big)G(\pnu)+2(|e_{\nu}(\rhonu)|^2+1)\frac{| \nabla \psi|^2}{\psi}.
\end{split}
\end{equation}
Plugging in $\nu^{-1}\big(\pnu
-\wnu\big)^2=\nu|\Delta\wnu|^2$,
and applying Holder's inequality to various terms, there must exist a constant $C>0$ such that
\[
\int_0^{\infty}\!\!\! \int_{\RR^d} \psi|\nabla \wnu|^2+\nu\psi|\Delta \wnu|^2\leq C N(\psi)\big(\norm{\pnu}_{L^{\infty}([0,\infty)\times\RR^d)}+\norm{e_{\nu}(\rhonu)}_{L^{\infty}([0,\infty)\times\RR^d)}^2+1\big)
\]
We then note that 
\[
-\frac{b_0}{a_0}(a_0-2a)_+\leq e_{\nu}(a)\leq \frac{b}{a_0}(2a-a_0)_+
\]
where $b=\inf \partial f_{\nu}(a)$ and $b_0=\inf \partial f_{\nu}(a_0/2)$.
Therefore, 
\[
|e_{\nu}(\rhonu)|\leq b_0+\frac{2}{a_0}\rhonu \pnu
\]
and the result now follows.
\end{proof}

One final interesting application of EEE is related to what is known in the literature as a complementarity condition.
Typically, the complementarity condition is a specific relation that must hold for sufficiently regular solutions $(\rho, p)$ to incompressible versions of Darcy's law.
For the flow $\partial_t \rho-\nabla \cdot (\rho\nabla p)=\rho G(p)$ with the incompressible coupling condition $p(1-\rho)=0$, the   complementarity condition takes the form
\[
\Delta p+G(p)=0 \quad \textup{almost everywhere in} \, \{p>0\}. 
\]
In our final corollary, we will show that EEE implies an approximate analogue of the complementarity condition for Brinkman's law as well.  Let us note that this was proven for the incompressible Brinkman model in \cite{PV, DEBIEC2020}.  However,  we will see that this condition appears in more general contexts beyond incompressibility.

\begin{cor}[The Brinkman complementarity condition]\label{cor:complementarity}
If there exists some $T>0$ and some value $V_{p}\in \RR$ such that 
 $\norm{\rhonu}_{L^{\infty}([0,T]\times\RR^d)}\leq \sup \partial f_{\nu}^*(V_{p})$,
then
\[
\Delta \wnu+G(\pnu)=0 \quad \textup{almost everywhere on} \; \{(t,x)\in [0,T]\times\RR^d: \pnu(t,x)>V_p\}.
\]  
and 
\[
-\Delta \wnu\leq \sgn_+(\pnu-V_p)G(\pnu) \quad \textup{almost everywhere on} \; \{(t,x)\in [0,T]\times\RR^d: \wnu(t,x)>V_p\}.
\] 
\end{cor}
\begin{remark}
   For the incompressible version of Brinkman's law, we must have $\rhonu\leq 1$ everywhere.  Since the dual of the incompressible energy, $f_\nu\equiv f_{0}$ in \eqref{graph},  is $f_{\nu}^*(b)=b_+$ and $1\in \partial f_{\nu}^*(0)$, we see that we can choose $V_p=0$, which gives the usual complementarity condition except that $\Delta p$ is replaced by $\Delta \wnu$. 
\end{remark}
\begin{proof}
Define $z_{\nu}'(b):=(b-V_p)_+$ and $e_{\nu}(a)=a\int_0^a \frac{z'_{\nu}(f'_{\nu}(\alpha))}{\alpha^2}\, d\alpha. $  It then follows from our choice of $V_p$, that $e_{\nu}(\rhonu)=0$ on $[0,T]\times\RR^d$.  If we choose some test function $\psi\in C^{\infty}_c([0,\infty)\times\RR^d)$ such that $\psi$ is supported on $[0,T]\times\RR^d$, then  from \eqref{eq:eee} and the vanishing of $e_{\nu}(\rhonu)$, we have
    \begin{equation*}
0=\int_0^{\infty}\!\!\! \int_{\RR^d} \psi \big(z_{\nu}'(\pnu
)-z_{\nu}'(\wnu)\big)\frac{\pnu-\wnu}{\nu}+\psi z_{\nu}''(\wnu)|\nabla \wnu|^2  -z_{\nu}(\wnu)\Delta \psi -z'_{\nu}(\pnu)G(\pnu).
\end{equation*}
Recalling that ${\pnu-\wnu}=-\nu \Delta \wnu$ and moving the Laplacian off of $\psi$ and back on to $z_{\nu}(\wnu)$, we find that
\begin{equation*}
0=\int_0^{\infty}\!\!\! \int_{\RR^d} \psi \Big( -\big(z_{\nu}'(\pnu
)-z_{\nu}'(\wnu)\big)\Delta \wnu-z'_{\nu}(\wnu)\Delta \wnu-z'_{\nu}(\pnu)G(\pnu)\Big).
\end{equation*}
Thus, 
\begin{equation*}
0=\int_0^{\infty}\!\!\! \int_{\RR^d} \psi z'_{\nu}(\pnu)\big(\Delta \wnu+G(\pnu)\big).
\end{equation*}
for every $\psi\in C^{\infty}_c([0,\infty)\times\RR^d)$ such that $\psi$ is supported on $[0,T]\times\RR^d$.  Hence, we see that $z'_{\nu}(\pnu)\big(\Delta \wnu+G(\pnu)\big)=0$ almost everywhere on $[0,T]\times \RR^d$,  which implies the first  result.

\smallskip

For the second result, we just need to establish what happens when $\wnu>V_p$ but $\pnu\leq V_p$.  In this case, the equation $\wnu-\nu\Delta \wnu=\pnu$ implies that $-\nu\Delta \wnu\leq 0$.  Hence, when $\wnu>V_p$ we have $-\Delta \wnu\leq \sgn_+(\pnu-V_p)G(\pnu)$. 
\end{proof} 

\section{Uniform \textit{a priori} estimates and compactness}\label{sec: a priori} 
\begin{lemma}\label{lem:p_rho_bounds}
The following bounds hold uniformly for all $\nu> 0$ sufficiently small
$$\pnu, \rhonu\in L^{\infty}([0,\infty)\times \RR^d), \quad \rhonu\in L^{\infty}_{\loc}([0,\infty);L^1(\RR^d)).$$
In particular, if we set $B_p:=\max\big(p_H\, ,\,  B\big)$, where $B$ is the constant identified in Definition \ref{def in} and $M(0):=\sup_{\nu\geq 0} \norm{\rhonu^{\mathrm{in}}}_{L^1(\RR^d)}$, then \[
0\leq p_\nu\leq B_p, \quad 0\leq \rhonu\leq \sup\partial f^*_{\nu}(B_p),\quad  \norm{\rhonu}_{L^{\infty}([0,T];L^1(\RR^d))}\leq e^{G(0)T}M(0),
\]
where we note that  $\limsup_{\nu\to 0} \sup\partial f^*_{\nu}(B_p)$ is finite.
\end{lemma}
\begin{proof}
    Define $z_{\nu}'(b):=(b-B_p)_+$ and $e_{\nu}(a)=a\int_{0}^a \frac{z'_{\nu}(f'_{\nu}(\alpha))}{\alpha^2}d\alpha$, where we note that the integral for $e_{\nu}$ is well defined since $z'_{\nu}(f'_{\nu}(\alpha))=0$ for all $\alpha\leq \sup\partial f^*_{\nu}(B_p)$.  Now we can apply Lemma \ref{lem:e_from_z},  to see that $e_{\nu}$ and $z_{\nu}'$ satisfy the coupling relation.  In turn,  Corollary \ref{cor:ede} yields that  $e_{\nu}(\rhonu)$ satisfies the total energy dissipation identity \eqref{eq:edi_T}. 
    Since Definition \ref{def in} guarantees that  $\rhonu^{\textup{in}}\leq \sup \partial f_{\nu}^*(B)$, it follows that $e_{\nu}(\rhonu^{\mathrm{in}})=0$.  Similarly, if $e_{\nu}(\rhonu)>0$, then $\pnu>B_p\geq p_H$.  Therefore, by Assumption~\ref{as: reaction}
    \[
    (e_{\nu}(\rhonu)+z'_{\nu}(\pnu))G(\pnu)\leq 0
    \]
     almost everywhere.  As a result, we can use the above considerations and Corollary~\ref{cor:dissipation} to obtain the inequality 
    \[
\begin{split}
 &\int_{\RR^d}  e_{\nu}(\rhonu(T))+\int_0^T\!\!\!\int_{\RR^d} { \big(z_{\nu}'(\pnu
)-z_{\nu}'(\wnu)\big)\frac{\pnu-\wnu}{\nu}}+ z_{\nu}''(\wnu)|\nabla \wnu|^2  \leq 0.
\end{split}
    \]
All of the terms are nonnegative, so each of the individual quantities on the left-hand side is zero. Automatically, this gives us $z'(\wnu)=0$, i.e. $\wnu\leq B_p$ almost everywhere.  We then see that   
   \[
   0=\int_0^T\!\!\!\int_{\RR^d} { \big(z_{\nu}'(\pnu
)-z_{\nu}'(\wnu)\big)\frac{\pnu-\wnu}{\nu}}\geq \int_0^T\!\!\!\int_{\RR^d} {z_{\nu}'(\pnu
)\frac{\pnu-B_p}{\nu}}=\int_0^T\!\!\!\int_{\RR^d} {\frac{(\pnu-B_p)_+^2}{\nu}},
\]
so $\pnu\leq B_p$ almost everywhere.  The bound $0\leq \rhonu\leq \sup\partial f^*_{\nu}(B_p)$ then follows from the coupling law $\rhonu\in \partial f_{\nu}^*(\pnu)$.  We then  note that 
\[
\sup\partial f^*_{\nu}(B_p)\leq f^*_{\nu}(B_p+1)-f^*_{\nu}(B_p),
\]
Thus, \[
\limsup_{\nu\to 0} \sup\partial f^*_{\nu}(B_p)\leq f^*_{0}(B_p+1)-f^*_{0}(B_p)<\infty.
\]
Note that here we are strongly using the fact that $f^*_{0}(b)\in \RR$ for all $b\in \RR$, which follows from Assumption \ref{as: energies} Part 2.

Finally, the $L^1$ bound $\norm{\rhonu}_{L^{\infty}([0,T];L^1(\RR^d))}\leq e^{G(0)T}M(0)$ follows readily from \eqref{eq: rho} and the fact that $G$ is decreasing in the pressure.
\end{proof}

\begin{lemma}\label{lemma:w_bounds}
The following holds uniformly in $\nu> 0$
\begin{itemize}
     \item[(i)] $\nabla \wnu \in L^2_{\loc}([0,\infty)\times \RR^d),$ 
 \item[(ii)] $\sqrt \nu \Delta  \wnu \in L^2_{\loc}([0,\infty)\times \RR^d),$ 
 \item[(iii)] $\partial_t \rhonu \in L^2_{\loc}([0,\infty); H^{-1}_{\loc}(\RR^d))$.
\end{itemize}
\end{lemma}
\begin{proof}
    From Corollary \ref{cor:w_control} and the uniform bounds on $\varrho_\nu$ and $\pnu$ we find \textit{(i)} and \textit{(ii)}. The uniform boundedness of $\nabla \wnu$ implies \textit{(iii)}. 
\end{proof}

The previous two Lemmas give us the following compactness properties.

\begin{lemma}\label{lem:compactness}
There exists $p, \rho, R\in L^{\infty}([0,\infty)\times \RR^d)$
and  $m\in L^2_{\loc}([0,\infty)\times\RR^d)$
such that 
\[
\nabla p\in L^2_{\loc}([0,\infty)\times\RR^d),\quad  \rho, R\in L^{\infty}_{\loc}([0,\infty);L^1( \RR^d)), \quad \partial_t \rho\in L^2_{\loc}([0,\infty);H^{-1}_{\loc}(\RR^d))
\]
and up to a subsequence we have for any $q\in [1,\infty)$ the following convergence properties as $\nu\to 0$
 
\begin{equation}
    \begin{split}
    \vspace{.1in} 
 |\pnu-\wnu| \rightarrow 0  \quad &\textup{strongly in } L^q_{\loc}([0,\infty)\times\RR^d),\\
    \vspace{.1in}
  \rhonu \rightharpoonup \rho, \quad \rhonu G(\pnu) \rightharpoonup R, \quad     &\textup{weakly in } L^q_{\loc}([0,\infty);L^q(\RR^d)),  \\
\vspace{.1in}
  \pnu \rightharpoonup p, \quad  \wnu \rightharpoonup p \quad     &\textup{weakly in } L^q_{\loc}([0,\infty)\times\RR^d),  \\
\vspace{.1in}
\nabla \wnu \rightharpoonup \nabla p, \quad  \rhonu\nabla \wnu \rightharpoonup m  \quad &\textup{weakly in } L^2_{\loc}([0,\infty)\times\RR^d), 
\end{split}
\end{equation}
In particular, this implies that along the selected subsequence, for each time $t\geq0$, the $\rhonu(t,\cdot)$ form a tight family of measures.
\end{lemma}
\begin{proof}
 To establish  the strong convergence of $|\pnu-\wnu|$, we recall that $\pnu-\wnu=\nu\Delta \wnu$ and from Lemma \ref{lemma:w_bounds} we know that $\nu^{1/2}\Delta\wnu$ is uniformly bounded in $L^2_{\loc}([0,\infty)\times\RR^d)$.  Hence, interpolating between the $L^2_{\loc}([0,\infty)\times\RR^d)$ bound on $\nu^{1/2}\Delta\wnu$ and the uniform $L^{\infty}$ bounds on $\pnu, \wnu$, the strong convergence follows.   Note that this automatically implies that $\wnu, \pnu$ must have the same weak limit.

All of the remaining weak convergence properties follow from the uniform bounds in Lemmas \ref{lem:p_rho_bounds} and \ref{lemma:w_bounds}, except for the convergence of $\rhonu$ and $\rhonu G(\pnu)$ in $ L^q_{\loc}([0,\infty);L^q(\RR^d))$ instead of  $L^q_{\loc}([0,\infty)\times\RR^d)$. To eliminate the need to localize the weak convergence in space, it suffices to show that for any time $T>0$ we have (along a subsequence)
\[
\limsup_{\nu \to 0} \int_{\RR^d} \rhonu(T)\leq \int_{\RR^d} \rho(T),
\]
which will imply that  $\rhonu(T,\cdot)$ form a tight family of measures for each time $T> 0$.

Integrating the evolution equation \eqref{eq: rho}, it follows that for any time $T>0$, 
\[
\int_{\RR^d} e^{-G(0)T}\rhonu(T)=\int_{\RR^d} \rhonu^{\textup{in}}+\int_0^T\!\!\!\int_{\RR^d} e^{-G(0)t}\rhonu (G(\pnu)-G(0)).
\]
We can derive a similar equation for $\rho$, as our convergence properties imply that for any $\vp\in C^{\infty}_c([0,\infty)\times\RR^d)$, the limit variables $(\rho, m, R)$ satisfy
\[
	\int_0^{\infty}\!\!\! \int_{\RR^d} m \cdot \nabla \varphi - \rho \partial_t \varphi  
          = \int_0^{\infty}\!\!\! \int_{\RR^d} \varphi R  +\int_{\RR^d} \varphi(0) \rho^{\mathrm{in}}.
\]
Since $\rho, R\in L^{\infty}_{\loc}([0,\infty);L^1(\RR^d))$ and $\rho^{\textup{in}}\in L^1(\RR^d)$, it then follows that the above equation holds for all $\vp\in C^{\infty}_c([0,\infty);C^{\infty}(\RR^d))$.
Therefore, we can deduce that
\[
\int_{\RR^d} e^{-G(0)T}(\rhonu-\rho)(T)=\int_{\RR^d} \rhonu^{\textup{in}}-\rho^{\textup{in}}+\int_0^T\!\!\!\int_{\RR^d} e^{-G(0)t}\big(\rhonu (G(\pnu)-G(0))-R+G(0)\rho\big).
\]
Given $\epsilon>0$, choose some compact set $K\subset\RR^d$ such that 
\[
\int_0^T\!\!\!\int_{\RR^d\setminus K} G(0)\rho +|R|<\epsilon.
\]
  Since we know that $\rhonu, \rhonu G(\pnu)$ converge weakly to $\rho, R$ in $L^1([0,T]\times K)$ and $\rhonu^{\textup{in}}$ converges weakly to $\rho^{\textup{in}}$ in $L^1(\RR^d)$, it follows that
  \begin{align*}
\limsup_{\nu\to 0}\int_{\RR^d} e^{-G(0)T}(\rhonu-\rho)(T)&= \limsup_{\nu\to 0} \int_0^T\!\!\!\int_{\RR^d\setminus K} e^{-G(0)t}\big(\rhonu (G(\pnu)-G(0))-R+G(0)\rho\big)\\
&\leq \epsilon+\limsup_{\nu\to 0} \int_0^T\!\!\!\int_{\RR^d\setminus K} e^{-G(0)t}\rhonu (G(\pnu)-G(0))\\
&\leq \epsilon,
\end{align*}
where the final inequality uses the fact that $G$ is decreasing.
\end{proof} 

Though nearly all of our convergence properties are weak, the next lemma shows that various nonlinear properties persist through the weak limits.  

\begin{lemma}\label{lem:nonlinear_weak}
As $\nu\to 0$, up to a subsequence 
\begin{equation}\label{eq:nonlinear_limits}
    \begin{array}{cc}
    \vspace{.1in}
\rhonu\pnu \rightharpoonup \rho p &\textup{weakly in } L^q_{\loc}([0,\infty);L^q(\RR^d)),\\
\vspace{.1in}
f_{\nu}(\rhonu) \rightharpoonup f_0(\rho) &\textup{weakly in } L^q_{\loc}([0,\infty);L^q(\RR^d)),\\
 f_{\nu}^*(\pnu),\; f_{\nu}^*(\wnu) \rightharpoonup f_0^*(p)  &\textup{weakly in } L^q_{\loc}([0,\infty);L^q(\RR^d)),
\end{array}
\end{equation}
where $q$ is any element of $[1,\infty)$.
In particular, $\rho p=f_0(\rho)+f_0^*(p)$.
\end{lemma}
\begin{proof}
Since $\pnu-\wnu$ converges strongly to zero in $L^q_{\loc}([0,\infty)\times\RR^d)$ and all of the quantities are uniformly bounded in $L^{\infty}([0,\infty)\times\RR^d)$, the weak convergence of $\rhonu\pnu$ to $\rho p$ in $L^q_{\loc}([0,\infty)\times \RR^d)$ will follow from the weak convergence of $\rhonu\wnu$ to $\rho p$. 
From Lemma~\ref{lemma:w_bounds} (iii), $\rhonu$ is uniformly bounded in $H^1_{\loc}([0,\infty);H^{-1}_{\loc}(\RR^d))$, therefore, 
$\rhonu$ is strongly compact in $C_{\loc}([0,\infty); H^{-1}_{\loc}(\RR^d))$.
Since $\wnu$ is uniformly bounded in $L^2_{\loc}([0,\infty);H^1_{\loc}\RR^d))$,  a compensated compactness argument implies that along a subsequence, $\wnu(\rhonu-\rho)$ converges weakly to zero in $L^q_{\loc}([0,\infty)\times\RR^d)$ for any $q\in [1,\infty)$.  We can then upgrade this to weak convergence in  $L^q_{\loc}([0,\infty);L^q(\RR^d))$ since the $\rhonu$ form a tight family of measures (along our selected subsequence).

\smallskip

To prove the remaining properties, we first note that the uniform bound on $\rhonu\pnu$ in $L^{\infty}([0,\infty)\times\RR^d)$ together with the duality relation $\rhonu\pnu=f_{\nu}(\rhonu)+f^*_{\nu}(\pnu)$ implies that $f_{\nu}(\rhonu)$ and $f^*_{\nu}(\pnu)$ are each uniformly bounded in $L^{\infty}([0,\infty)\times\RR^d)$ and form a tight family of measures along an appropriate subsequence.  Therefore, after passing to another subsequence, $f_{\nu}(\rhonu)$ and $f^*_{\nu}(\pnu)$ must converge weakly in $L^q_{\loc}([0,\infty);L^q(\RR^d))$ to limits $A$ and $A^*$ respectively.   
The duality relation $\rhonu\pnu=f_{\nu}(\rhonu)+f^*_{\nu}(\pnu)$ and the weak convergence of $\rhonu\pnu$ to $\rho p$ implies that $A+A^*=\rho p$ and then we can apply Young's inequality to see that $A+A^*\leq f_0(\rho)+f_{0}^*(p)$ almost everywhere.  On the other hand, the convexity of $f_{\nu}, f^*_{\nu}$ and their respective convergence to $f_0, f^*_0$, implies that $A\geq f_0(\rho)$ and $A^*\geq f^*_0(p)$ almost everywhere. This is only possible if $A=f_0(\rho)$ and $A^*=f^*_0(p)$. 

\smallskip

 By Lemma~\ref{lem:compactness}, $|w_\nu - p_\nu|\to 0$ strongly, therefore
$|f_{\nu}^*(\wnu)-f^*_{\nu}(\pnu)|\leq \sup\partial f^*_{\nu}(B_p)|\wnu-\pnu|$ implies that $|f_{\nu}^*(\wnu)-f^*_{\nu}(\pnu)|$ converges strongly to zero in $L^q_{\loc}([0,\infty)\times\RR^d)$. Thus, $f^*_{\nu}(\wnu)$ converges to $f_0(p)$ weakly in $L^q_{\loc}([0,\infty)\times\RR^d)$.  This can then be upgraded to weak convergence in $L^q_{\loc}([0,\infty);L^q(\RR^d))$ by noting that for any $T\geq 0$, we can use convexity to get
\[
\int_0^T\!\!\!\int_{\RR^d} f^*_{\nu}(\wnu)\leq \int_0^T\!\!\!\int_{\RR^d} f^*_{\nu}(\pnu)-f^{*\prime}_{\nu}(\wnu)(\pnu-\wnu)=\int_0^T\!\!\!\int_{\RR^d} f^*_{\nu}(\pnu)-\nu f^{*\prime\prime}_{\nu}(\wnu)|\nabla \wnu|^2.
\]
So the tightness of $ f^*_{\nu}(\pnu)$ implies the tightness of $ f^*_{\nu}(\wnu)$ (note if $f^*_{\nu}$ is not twice differentiable, one can first approximate $f^*_{\nu}$ from below by a $C^2$ convex function, run the above argument, and then use monotone convergence to deduce the tightness property).
\end{proof}

\section{Identification of the limit equation and strong compactness}
\label{sec: strong comp} 
From the compactness results of Lemma~\ref{lem:compactness}, we can pass to the limit in \eqref{eq: rho} to show that $\rho, m, p, R$ satisfy the following equations in the weak sense 
\begin{equation}\label{eq:m}
    \partial_t \rho - \nabla \cdot  m = R, \quad \rho\in \partial f_0(p),
\end{equation}
where we note that the condition $\rho\in \partial f_0(p)$ is equivalent to the condition $\rho p=f_0(\rho)+f^*_0(p)$ established in Lemma \ref{lem:nonlinear_weak}.
In the first part of this Section, we will focus on proving that $m=\rho\nabla p$.  Once this is established, we will be able to use the $H^{-1}$ energy dissipation relation (Corollary~\ref{cor:h-1}) to prove that $\nabla f^*_{\nu}(\wnu)$ converges strongly to $\nabla f_0^*(p)$.  From there, we will show that $\nabla \wnu$ converges strongly to $\nabla p$ and hence $R=G(p)$, which completes the proof of our main theorem.

\subsection{Proving $m=\rho\nabla p$}
\label{sec: m} 

Let us now expose our strategy to prove that $m=\rho \nabla p$, and highlight its relation to our family of energy evolution equations.  Since $\rho\nabla p=\nabla (f_0^*(p))$ and by Lemma \ref{lem:nonlinear_weak} $f_{\nu}^{*}(\wnu)$ converges weakly to $f_0^*(p)$, it is enough to show that $$
\rhonu\nabla \wnu-\nabla (f_{\nu}^*(\wnu))=(\rhonu-f^{*\prime}_{\nu}(\wnu))\nabla \wnu \rightharpoonup 0
$$ in the sense of distributions.  We will actually prove a  strong convergence of the difference to zero in $L^1_{\loc}([0,\infty)\times\RR^d)$.  To see why this quantity may be small, recall that 
$\pnu-\wnu$ converges strongly to zero as $\nu\to 0$ in $L^q_{\loc}([0,\infty)\times\RR^d)$ (Lemma \ref{lem:nonlinear_weak}). 
Hence $\rhonu\in \partial f_{\nu}^*(\pnu)$ and $f^{*\prime}_{\nu}(\wnu)$ should be close if $f^*_0$ is $C^1$.  

But in our application $f^*$ may not be $C^1$. Indeed for the joint limit of Brinkman to Darcy, $f$ is given by \eqref{pme} and  $f^*_{\nu}(b)=b_+^{1+\gamma_{\nu}^{-1}}$ with $\gamma_{\nu}\to \infty$ as $\nu\to 0$. Hence neither $f^*_{\nu}$ and  $f^*_0$ is $C^1$: this lack of differentiability has been the main obstacle for rigorously establishing the joint limit.

\medskip

On the other hand, there is still hope, as convexity of $f^*_{\nu}$ implies that the singular set $S_{\nu}:=\{b: f^*_{\nu} \hbox{ is not } C^1\}$ has measure zero. While it is possible that $\pnu, \wnu$ could concentrate on  $S_{\nu}$, Sobolev regularity of $w_{\nu}$ implies that for any compact set $K\subset [0,\infty)\times\RR^d$,  $|\nabla \wnu|$ gives no mass to $\{(t,x)\in K: \wnu(t,x)\in S_{\nu} \}$ (see for instance \cite{evans_gariepy} Chapter 4). Hence, when we consider the product $(\rhonu-f^{*\prime}_{\nu}(\wnu))\nabla \wnu$, one could at least hope that one of the two terms will stay small almost everywhere. Nonetheless, turning the above wishful intuition into a rigorous argument, in particular one with quantitative control in terms of $\nu>0$, has remained a significant challenge in the literature. 

\medskip

Fortunately, the aforementioned reasoning can be quantified through an appropriate invocation of Proposition \ref{prop:eee} using the EEE  \eqref{eq:eee} therein. More precisely, for any $A\subset \RR$ and for any compact set $K\subset [0,\infty) \times \RR^d$, Proposition \ref{prop:w_gradient_control}, yields the estimate
\begin{equation}\label{est:uniform}
 \iint_{K} \chi_A(\wnu)|\nabla \wnu|^2\leq C |A|,
\end{equation}
where $C=C(K)$ is independent of $\nu$. 
 To use the above bound, note that if we fix a value of $\delta>0$ and take $A=S_{\nu,\delta}:=\{b\in \RR: \sup_{\beta\in [0,B_p]} \frac{|f^{*\prime}_{\nu}(b)-f^{*\prime}_{\nu}(\beta)|}{|b-\beta|}>\frac{1}{\delta} \}$, then a standard maximal function type covering argument implies that $|S_{\nu,\delta}|\lesssim \delta$ (c.f. for instance \cite{stein})  and on the complement of $S_{\nu,\delta}$ we can make $(\rhonu-f^{*\prime}_{\nu}(\wnu))$ small.  After interpolating with an appropriate value of $\delta$, we will be able to show in Proposition \ref{prop:density_swap} that $(\rhonu-f^{*\prime}_{\nu}(\wnu))\nabla \wnu$ converges strongly in $L^1_{\loc}([0,\infty)\times\RR^d)$ to $0$.

\begin{prop}\label{prop:w_gradient_control}
   Let $z_{\nu}:\RR\to\RR$ be a proper, lower semicontinuous, convex function with $z_{\nu}\in W^{2,1}_{\loc}([0,\infty))$ and $z_{\nu}(0)=z_{\nu}'(0)=0$.
  Given any nonnegative test function $\psi\in C^{\infty}([0,\infty)\times\RR^d)$ with compact support in time such that $N(\psi)<\infty$, we have
  \begin{equation*}
 \int_0^{\infty}\!\!\! \int_{\RR^d}  \psi z''_{\nu}(\wnu)|\nabla \wnu|^2\lesssim  N(\psi)\sup_{b\in [0,B_p]} |z'_{\nu}(b)|,
\end{equation*}
where we recall that $N(\psi)$ is defined in \eqref{eq:n_psi} and $B_p$ is the pressure bound from Lemma \ref{lem:p_rho_bounds}. 
\end{prop}
\begin{proof}
   If we set $e_{\nu}(a)=\int_{a_0}^a \frac{z'_{\nu}(f'_{\nu}(\alpha))}{\alpha^2}\, d\alpha$, then $z_{\nu}$ and $e_{\nu}$ satisfy the coupling relation and we can use Proposition \ref{prop:eee} to deduce the inequality 
    \begin{align*}
\int_0^{\infty}\!\!\! \int_{\RR^d}  \psi z''_{\nu}(\wnu)|\nabla \wnu|^2&\leq    \int_{\RR^d} \psi(0) e_{\nu}(\rhonu^{\mathrm{in}})+\int_0^{\infty}\!\!\! \int_{\RR^d} \psi \big( e_{\nu}(\rhonu)+z'_{\nu}(\pnu)\big)G(\pnu)\\[0.3em]
&\qquad -\int_0^{\infty}\!\!\! \int_{\RR^d}  e_{\nu}(\rhonu)(\nabla \wnu\cdot \nabla \psi-\partial_t \psi)-\int_0^{\infty}\!\!\! \int_{\RR^d}z_{\nu}'(\wnu)\nabla \wnu\cdot \nabla \psi.
\end{align*}
It is clear, by convexity that  
\[
|e_{\nu}(a)|\leq z'_{\nu}\big(f'_{\nu}(\max(a,a_0))\big)\left|1-\frac{a}{a_0}\right|.
\]
Therefore, using our uniform $L^{\infty}$ control on $\rhonu, \pnu, \wnu$ from Lemma \ref{lem:p_rho_bounds}, it follows that
\[
|e_{\nu}(\rhonu)|\leq \left(\frac{\rhonu}{a_0}+1\right)\sup_{b\in [0,B_p]} |z'_{\nu}(b)|,
\]
and hence,
\begin{equation*}
\int_0^{\infty}\!\!\! \int_{\RR^d}  \psi z''_{\nu}(\wnu)|\nabla \wnu|^2\lesssim  \sup_{b\in [0,B_p]}|z'_{\nu}(b)| \bigg(\int_{\RR^d} \psi
(0) +\int_0^{\infty}\!\!\! \int_{\RR^d} \psi +|\nabla \wnu\cdot \nabla \psi-\partial_t \psi|+|\nabla \wnu\cdot \nabla \psi|\bigg).
\end{equation*}
Using Corollary \ref{cor:w_control} and Lemma \ref{lem:p_rho_bounds} again, we conclude. 
\end{proof}

\begin{prop}\label{prop:density_swap}
  Given any nonnegative test function $\psi\in C^{\infty}([0,\infty)\times\RR^d)$ with compact support in time such that $N(\psi)<\infty$, we have
\begin{equation*}
 \int_0^{\infty}\!\!\! \int_{\RR^d}  \psi \big|\rhonu-f_{\nu}^{*\, \prime}(w_{\nu})||\nabla \wnu|\lesssim \nu^{1/6} N(\psi)
\end{equation*}
where we recall that $N(\psi)$ is defined in \eqref{eq:n_psi}. In particular, this implies that $m=\rho\nabla p$ in  \eqref{eq:m}.

\end{prop}

\smallskip
 
\begin{remark}   Note that when $f^*_{\nu}$ is uniformly Lipschitz, one  would obtain a much better convergence rate of $\nu^{1/2}$ instead of $\nu^{1/6}$.
\end{remark}
\begin{remark}
    It is likely that the generic rate could be improved to $\nu^{1/4}$ using the entropy dissipation relation in  Appendix~\ref{app: eee powers}.
\end{remark}

\begin{proof}

Let $B_p$ be the uniform $L^{\infty}([0,\infty)\times\RR^d)$ bound on $\wnu, \pnu$ from Lemma \ref{lem:p_rho_bounds}.
Fix some $\delta>0$ and define 
\[
S_{\nu,\delta}:=\left\{b\in [0,B_p]: \sup_{\beta\in [0,B_p]} \frac{|\partial f_{\nu}^*(b) -\partial f_{\nu}^*(\beta) | }{|b-\beta|}>\frac{1}{\delta}\right\},
\]
where we interpret
\[
|\partial f_{\nu}^*(b) -\partial f_{\nu}^*(\beta) |=\max\Big( |\sup \partial f_{\nu}^*(b)-\inf \partial f_{\nu}^*(\beta)|\, , \,  |\sup \partial f_{\nu}^*(\beta)- \inf \partial f_{\nu}^*(b)|\Big).
\]
Note that $\sup_{\beta\in [0,B_p]} \frac{|\partial f_{\nu}^*(b) -\partial f_{\nu}^*(\beta) | }{|b-\beta|}$ is the maximal function of the measure valued second derivative of $f_{\nu}^*$.
A standard maximal function argument then implies that there exists a constant $C>0$ such that
\[
|S_{\nu,\delta}|\leq C\delta \sup\partial f^*_{\nu}(B_p). 
\]
Now we can estimate
\[
\int_0^{\infty}\!\!\! \int_{\RR^d}  \psi \big|\rhonu-f_{\nu}^{*\, \prime}(w_{\nu})\big||\nabla \wnu|\leq \delta^{-1}I_1(\nu)+I_2(\nu,\delta)^{1/2}\norm{\psi}_{L^1([0,\infty)\times \RR^d)}^{1/2}\norm{\rhonu}_{L^{\infty}([0,\infty)\times\RR^d)}^{1/2}
\]
where 
\[
 I_1(\nu):=\int_0^{\infty}\!\!\! \int_{\RR^d} \psi|\pnu-\wnu||\nabla\wnu|,\quad I_2(\nu,\delta):=\int_0^{\infty}\!\!\! \int_{\RR^d} \psi \chi_{S_{\nu,\delta}}(\wnu)|\nabla w_{\nu}|^2.
\]
Recalling that $\pnu-\wnu=-\nu\Delta \wnu$, we have the straightforward estimate
\[
I_1(\nu)\leq \nu^{1/2}\norm{\rhonu}_{L^{\infty}([0,\infty)\times\RR^d)}\norm{\psi\nabla \wnu}_{L^2([0,\infty)\times\RR^d)}\norm{\nu^{1/2}\psi\Delta \wnu}_{L^2([0,\infty)\times\RR^d)}.
\]
We can then use Corollary \ref{cor:w_control} and the $L^{\infty}$ bounds from Lemma \ref{lem:p_rho_bounds}, to obtain the estimate
\[
I_1(\nu)\lesssim \nu^{1/2}  N(\psi).
\]
For $I_2$, we can use Proposition \ref{prop:w_gradient_control} with $z_{\nu}'' = \chi_{S_{\nu}, \delta}$, to get
\[
I_2(\nu,\delta)\lesssim N(\psi)\sup|z'_{\nu}| \lesssim N(\psi) |S_{\nu,\delta}|\lesssim \delta N(\psi).  
\]
Combining our two estimates, it follows that
\[
\int_0^{\infty}\!\!\! \int_{\RR^d}  \psi \big(\rhonu-f_{\nu}^{*\, \prime}(w_{\nu})\big)_+|\nabla \wnu|\lesssim N(\psi)\left( \frac{\nu^{1/2}}{\delta} + \delta^{1/2}\right) .
\]
Up to constants, the optimal choice for $\delta$ is $\delta=\nu^{1/3}$, which then gives the desired bound.

For the final claim, we note that the above estimate implies that $\rhonu\nabla \wnu-\nabla (f^*_{\nu}(\wnu))$ converges strongly to zero in $L^1_{\loc}([0,\infty)\times\RR^d)$.  On the other hand, $\nabla (f^*_{\nu}(\wnu))$ converges weakly in duality with $C^{\infty}_c([0,\infty)\times\RR^d)$ test functions to $\nabla (f_0^*(p))$.  From the duality relation $\rho p=f_0(\rho)+f^*_0(p)$, we must have $\nabla (f^*_0(p))=\rho\nabla p$ almost everywhere.  Thus, $m=\rho\nabla p$.
\end{proof}

\subsection{Strong compactness}

\label{sec: final}

Since we now know that $m=\rho\nabla p$, it follows from equation \eqref{eq:m} that $\rho$ and $p$ are weak solutions to the equation 
    \begin{equation}\label{eq: limit new}
    \partial_t \rho-\nabla \cdot (\rho\nabla p)=R, \quad \rho\in \partial f_0(p).
    \end{equation}
To conclude the proof of Theorem~\ref{theorem} it now remains to establish the strong convergence of $\nabla \wnu$ to $\nabla p$, which in particular will yield that $R=\rho G(p)$ (Proposition~\ref{prop:strong}). 
\smallskip

An important element of the proof is  $H^{-1}$ energy dissipation equation, Corollary~\ref{cor:h-1}, which yields strong convergence of $\nabla (f^*_{\nu}(\wnu))$  to $\nabla (f^*_0(p))$ (Proposition~\ref{prop:nablafstar}). From here, the strong convergence of $\nabla \wnu$ to $\nabla p$ will follow from the invertibility of $f^*_0$ on $(0,\infty)$. 
 
 \medskip

\begin{prop}\label{prop:nablafstar}
    $\nabla f^*_{\nu}(w_{\nu})$ converges strongly to $\nabla f^*_0(p)$ in $L^2_{\loc}((0,\infty); L^2(\RR^d))$.
\end{prop}

\begin{proof} Choose some time $T>0$.
If we plug in the choice $\eta(t)=t$ into Corollary~\ref{cor:h-1}, we have
\begin{equation}\label{ine 1}
\begin{split} 
\int_{\RR^d} T h_{\nu}(\rhonu(T)) + \int_0^T \!\!\!\int_{\RR^d} t |\nabla f_\nu^*(\wnu)|^2 -h_{\nu}(\rhonu)
&\leq \int_0^T\!\!\! \int_{\RR^d} t \rhonu G(p_\nu)f^*_{\nu}(\pnu),
\end{split}
\end{equation}
where $h_\nu (a)=a f_\nu(a)-2\int_0^a f_\nu(\alpha)\dx \alpha.$  
Set  $h_0(a):=af_0(a)-2\int_0^a f_0(\alpha)\dx \alpha$.
Observe that the analogue of \eqref{ine 1} also holds for the limit variables $\rho$ and $p$ with $h_0$, see for instance \cite{Jac21}:
   \begin{equation}\label{eq 1}
\int_{\RR^d} T h_0(\rho(T))+\int_0^T\!\!\!\int_{\RR^d} t|\nabla f^*_0(p)|^2-h_0(\rho)= \int_0^T\!\!\!\int_{\RR^d} t R f^*_0(p).
\end{equation}
Now we would like to compare \eqref{eq 1} to \eqref{ine 1} in the limit $\nu\to 0$.
By weak lower semicontinuity, it is clear that 
\[
\int_{\RR^d} T h_0(\rho(T)) \leq \liminf_{\nu\to 0} \int_{\RR^d} T h_{\nu}(\rhonu(T)).
\]
However, we would also like to show 
\begin{equation}\label{eq:eta_t_goal}
\int_0^T\!\!\!\int_{\RR^d} -h_{0}(\rho)\leq \liminf_{\nu\to 0} \int_0^T\!\!\!\int_{\RR^d} -h_{\nu}(\rhonu),
\end{equation}
which does not follow from weak lower semicontinuity.
To establish \eqref{eq:eta_t_goal}, we recall from the arguments in Corollary \ref{cor:h-1} that $h_{\nu}(\rhonu)=\rhonu f^*_{\nu}(\pnu)-h^*_{\nu}(f^*_{\nu}(\pnu))$.  Since $f^*_{\nu}(\pnu)$ converges weakly to $f^*_0(p)$ in $L^1_{\loc}([0,\infty);L^1(\RR^d))$, the convexity of $h^*_{\nu}$ implies that the weak $L^1_{\loc}([0,\infty);L^1(\RR^d))$ limit of $-h^*_{\nu}(f^*_{\nu}(\pnu))$ is below $-h^*_0(f^*_0(p))$.
Repeating the compensated compactness argument used in Lemma \ref{lem:growth_weak_limit}, we find that $\rhonu f^*_{\nu}(\pnu)$ converges weakly in $L^1_{\loc}([0,\infty);L^1(\RR^d))$  to $\rho f_0^*(p)$. Thus, \eqref{eq:eta_t_goal} holds.

It remains to show that 
\begin{equation}\label{eq101}
\limsup_{\nu\to 0} \int_0^T\!\!\! \int_{\RR^d} t \rhonu G(p_\nu)f^*_{\nu}(\pnu)\leq \int_0^T\!\!\! \int_{\RR^d} t R f^*_{0}(p).
\end{equation}
which will be a particular case of Lemma \ref{lem:growth_weak_limit} below. 
\medskip

Putting together \eqref{ine 1}, \eqref{eq 1}, \eqref{eq:eta_t_goal} and \eqref{eq101}, it follows that 
\[
\limsup_{\nu\to 0} \int_0^T\!\!\!\int_{\RR^d} t|\nabla f^*_{\nu}(\wnu)|^2 \leq \int_0^T\!\!\!\int_{\RR^d} t|\nabla f^*_0(p)|^2.
\]
 The weak convergence of $f^*_{\nu}(\wnu)$ to $f^*_0(p)$ combined with the above upper semicontinuity property is enough to deduce that $\nabla f^*_{\nu}(\wnu)$ converges strongly to $\nabla f^*_0(p)$ in $L^2_{\loc}((0,T];L^2(\RR^d))$, which completes the argument.  
\end{proof}

The lemma below shows \eqref{eq101} in the particular case $l_{\nu} = f_{\nu}^*$, $l_0=f_0^*$.

\begin{lemma}\label{lem:growth_weak_limit}
If  $\ell_{\nu}:\RR\to\RR$ is a sequence of functions that converge uniformly on compact subsets of $\RR$ to a continuous increasing function $\ell_0:\RR\to\RR$ and $\ell_{\nu}(\pnu)$ converges weakly $L^1_{\loc}([0,\infty)\times\RR^d)$ to a limit $\zeta$ such that $\zeta\leq \ell_0(p)$ almost everywhere, then for any nonnegative test function $\psi\in L^{\infty}_c([0,\infty);L^{\infty}(\RR^d))$
    \[
    \limsup_{\nu\to 0} \int_0^{\infty}\!\!\! \int_{\RR^d}\psi \rhonu G(\pnu) (\ell_{\nu}(\pnu)-\ell_0(p))\leq 0.
    \] 
\end{lemma}
\begin{proof}
    Thanks to our assumption on the convergence of $\ell_{\nu}$ to $\ell_0$, it will suffice to prove that
     \[
    \limsup_{\nu\to 0} \int_0^{\infty}\!\!\! \int_{\RR^d} \psi \rhonu G(\pnu) (\ell_{0}(\pnu)-\ell_{0}(p))\leq 0.
    \]
    Since $G$ is decreasing and $\ell_0$ is increasing, we have $(G(\pnu)-G(p))(\ell_0(\pnu)-\ell_0(p))\leq 0$ almost everywhere. Thus, it is enough to show that
    \[
    \limsup_{\nu\to 0} \int_0^{\infty}\!\!\! \int_{\RR^d}\psi \rhonu G(p) (\ell_{0}(\pnu)-\ell_{0}(p))\leq 0,
    \]
     which will follow if we can show that $\rhonu \ell_0(\pnu)$ converges weakly to   $\rho\zeta$.  Since $\ell_0$ is continuous, $\ell_0(\pnu)$ is spatially equicontinuous, thus, compensated compactness implies that $\rhonu \ell_0(\pnu)$ converges weakly to $\rho \zeta$.
     \end{proof}

\begin{prop}\label{prop:strong}
 Up to a subsequence, $\wnu\to p$ strongly in $L^2_{\loc}((0,\infty); H^1_{\loc}(\RR^d))$ and $\pnu\to p$ strongly in $L^2_{\loc}([0,\infty)\times\RR^d)$ as $\nu\to 0$. In particular $R=\rho G(p)$.  
\end{prop}
\begin{proof}
Proposition~\ref{prop:nablafstar} yields that $f^*_{\nu}(\wnu)\to f^*_0(p)$ in $L^2_{\loc}((0,\infty);H^1(\RR^d))$.  
Since $f^*_0$ is invertible on $(0,\infty)$, it follows that $\wnu\to p$ in $L^2_{\loc}([0,\infty)\times \RR^d)$ and hence Lemma \ref{lem:compactness} gives the convergence of $\pnu$ to $p$ in $L^2_{\loc}([0,\infty)\times\RR^d)$.

It remains to show that $\nabla \wnu$ converges strongly in $L^2_{\loc}([0,\infty)\times\RR^d)$ to $\nabla p$.  Note that, due to its weak convergence, it suffices to show that 
\[
\limsup_{\nu\to 0}\iint_K |\nabla \wnu|^2\leq \iint_K |\nabla p|^2.
\]
for any compact set $K\subset [0,T]\times \RR^d.$
  Given some $\delta>0$ and some $\beta>0$, let us define
\[
f_{\nu,\delta}^*(b):=\inf_{\theta\in \RR} f^*_{\nu}(\theta)+\frac{1}{2\delta}|\theta-b|^2,
\]  
and $K_{\nu, \beta}:=\{(t,x)\in K: \wnu(t,x)>\beta\}$.
Since $f_{\nu,\delta}^{*\prime}(b)$ is lesser or equal than any element in $ \partial f^*_{\nu}(b)$ for all $b$,  it follows that
\[
 |\nabla \wnu|^2\leq  \frac{1}{f_{\nu,\delta}^{*\prime}(\wnu)^2}|\nabla f^*_{\nu}(\wnu)|^2 \quad\hbox{ a.e. on } K_{\nu, \beta}.
\]
Thus, using the strong convergence of $\wnu$ to $p$ and the above estimate, we have 
\begin{align*}
\limsup_{\nu\to 0}\iint_K |\nabla \wnu|^2&\leq  \limsup_{\nu\to 0} \iint_{K_{\nu, \beta}} \frac{1}{f_{\nu,\delta}^{*\prime}(\wnu)^2}|\nabla f^*_{\nu}(\wnu)|^2+\iint_{K\setminus K_{\nu, \beta}} |\nabla f^*_{\nu}(\wnu)|^2\\[0.5em]
&=\iint_{\{(t,x)\in K: p(t,x)>\beta\}} \frac{1}{f_{0,\delta}^{*\prime}(p)^2}|\nabla f^*_0(p)|^2+\limsup_{\nu\to 0} \iint_{K\setminus K_{\nu,\beta}} |\nabla f^*_{\nu}(\wnu)|^2.
\end{align*}
After sending $\delta\to 0$, it follows that for any $\beta>0$
\[
\limsup_{\nu\to 0}\iint_K |\nabla \wnu|^2\leq \iint_K |\nabla p|^2+\limsup_{\nu\to 0} \iint_{K\setminus K_{\nu,\beta}} |\nabla f^*_{\nu}(\wnu)|^2.
\]
Finally, if we take $z''_{\nu,\beta}:=\chi_{[0,\beta]}$, then by Proposition \ref{prop:w_gradient_control}
\[
\lim_{\beta\to 0} \limsup_{\nu\to 0} \iint_{K\setminus K_{\nu,\beta}} |\nabla f^*_{\nu}(\wnu)|^2\lesssim \lim_{\beta\to 0}  N(\psi) \sup_{b\in [0,B_p]} z'_{\nu,\beta}(b)=0,
\]
for any $\psi\in C_c^{\infty}([0,T]\times\RR^d)$ such that $\psi=1$ on $K$.  Thus
\[
\limsup_{\nu\to 0}\iint_K |\nabla \wnu|^2\leq \iint_K |\nabla p|^2,
\]
and we conclude.
\end{proof}

Proposition~\ref{prop:strong} and the weak convergence of density $\rhonu$ from Lemma~\ref{lem:compactness} re-affirm the conclusion of Proposition~\ref{prop:density_swap}, namely that $m=\rho\nabla p$ in \eqref{eq:m}. In particular, we can conclude that $\rho$ solves
the continuity equation
$$
\partial_t \rho - \nabla\cdot(\rho\nabla p) = \rho G(p),
$$
along with the density-pressure coupling condition $p\in \partial f_0(\rho)$.  Since solutions to the above equations are unique \cite{PQV}, it follows that every convergent subsequence must converge to the same solution.
This concludes the proof of Theorem~\ref{theorem}.

\section*{Appendix}

 \appendix
 
\section{Multi-species cross-diffusion systems}
\label{app: sys}
In recent years models of the kind \eqref{eq: rho} that include more than one species of tissue cells (e.g. tumor cells, healthy issue, quiescient cells, ...) have attracted a lot of interest. In particular, several works have been focusing on how to prove the existence of weak solutions and study singular limits for coupled reaction-(nonlinear)-diffusion equations both of Darcy's and Brinkman's type \cite{CFSS, GPS, DEBIEC2020, DeSc}. Methods based on energy dissipation (in)equalities have been originally applied to this kind of tissue growth models expressly to deal with systems of coupled equations of the type \eqref{sys}, see \cite{KM18, LX2021, Jac21, Dav23, price2020global} for the incompressible limit and existence of weak solutions (for $\nu=0$), and \cite{DDMS, ES2023, DS2024} for the inviscid limit. Indeed, as it was first remarked in \cite{CFSS, GPS}, systems of coupled equations where the velocity field related to the total density-based pressure can exhibit sharp interfaces between the support of the two species. Due to the difficulty in establishing \textit{a priori} estimates which could lead to the strong compactness of the single species, major effort was employed to prove the strong compactness of the velocity field, hence the development of energy methods.

Here we claim that the main result of our paper also holds for systems of the kind 
\begin{equation}\label{sys}
 \left\{   \begin{array}{rl}
        \partial_t \rho^{(1)}_\nu\!\!\!&= \nabla \cdot (\rho^{(1)}_\nu \nabla \wnu) +\rho^{(1)}_\nu G^{(1)}(\pnu),\\[0.5em]
          \partial_t \rho^{(2)}_\nu\!\!\!&= \nabla \cdot (\rho^{(2)}_\nu \nabla \wnu) +\rho^{(2)}_\nu G^{(2)}(\pnu),
    \end{array}
    \right.
\end{equation}
with 
\begin{align*}
   -\nu \Delta \wnu + \wnu = \pnu, \qquad \pnu \in \partial f_\nu(\rhonu), \qquad \rhonu= \rhonu^{(1)}+\rhonu^{(2)}.
\end{align*}
We assume that both reaction terms $G^{(1)}, G^{(2)}$ satisfy Assumption~\ref{as: reaction}, while the sequence $\{f_\nu\}_\nu$ satisfies Assumption~\ref{as: energies}.
Then, solutions to \eqref{sys} converge to solution to
\begin{equation}\label{limit sys}
 \left\{   \begin{array}{rl}
        \partial_t \rho^{(1)}\!\!\!&= \nabla \cdot (\rho^{(1)} \nabla p) +\rho^{(1)} G^{(1)}(p),\\[0.5em]
          \partial_t \rho^{(2)}\!\!\!&= \nabla \cdot (\rho^{(2)} \nabla p) +\rho^{(2)} G^{(2)}(p),
    \end{array}
    \right.
\end{equation}
with 
$$p\in \partial f_0(\rho), \quad \rho=\rho^{(1)}+\rho^{(2)}.$$

\begin{assumption}\label{as: initial data sys}
    We consider a family of initial data $(\rhonu^{(i),\mathrm{in}})_{\nu>0}$ for $i=1,2,$ that are nonnegative, uniformly bounded in $L^1(\RR^d)\cap L^{\infty}(\RR^d)$. We assume that the initial data are {\it well-prepared}, namely that there exist $\rho^{(i),\mathrm{in}}\in L^1(\mathbb{R}^d)\cap L^\infty(\RR^d)$ such that  
  $\rho^{(i),\mathrm{in}}_\nu \rightharpoonup \rho^{(i),\mathrm{in}}$ weakly in $L^1(\mathbb{R}^d)$ for $i=1,2$ and there exists $B>0$ such that $\rho_{\nu}^{\mathrm{in}}\in \partial f_{\nu}^*(B)$
where $\rho^{\mathrm{in}}_\nu:=\rho^{(1),\mathrm{in}}_\nu+\rho^{(2),\mathrm{in}}_\nu$. 
\end{assumption}
  
\begin{theorem}\label{theorem sys}
    Let $\{\rho_\nu^{(i)}\}_{\nu> 0}$, $i=1,2,$ be a sequence of weak solutions to \eqref{sys} with the well-prepared initial data $\rhonu^{(i),\textup{in}}\in L^1(\RR^d)\cap L^{\infty}(\RR^d)$ as described above. Then the following holds for $f_{\nu}$ and $G^{(i)}$ satisfying Assumptions \ref{as: reaction}-\ref{as: energies}: there exists $\rho^{(i)}, p\in L^\infty(0,T;L^1(\RR^d)\cap L^\infty(\RR^d))$ such that, up to a subsequence
    \begin{align*}
         \rhonu^{(i)} &\rightharpoonup \rho^{(i)}, \text{ weakly in } L^2(0,T; L^2(\RR^d)),\\[0.5em]
         \pnu &\to p,\text{ strongly in } L^2(0,T; L^2_{\mathrm{loc}}(\RR^d)),\\[0.6em]
     \nabla \wnu &\to \nabla p, \text{ strongly in } L^2(0,T; L^2_{\mathrm{loc}}(\RR^d)),
    \end{align*}
 and $\rho^{(i)}, p$ is a solution to \eqref{limit sys} with initial data $\rho^{(i),\mathrm{in}}$.  
\end{theorem}

It is not among the purposes of this paper to give a fully detailed proof of this result. However, since the argument closely follows the one detailed in the previous sections, let us give a brief explanation to support our claim.

Let us recall that in the one-species case the entire argument lies on the energy evolution equation \eqref{eq:eee}. For system \eqref{sys}, let us then consider the equation on the sum $\rhonu=\rhonu^{(1)}+\rhonu^{(2)}$ 
 $$\partial_t \rho_\nu =\nabla \cdot (\rhonu \nabla \wnu) +\rhonu^{(1)} G^{(1)}(\pnu) +\rhonu^{(2)} G^{(2)}(\pnu),$$
where the only difference to \eqref{eq: rho} is the reaction term.
Therefore, the Energy Evolution Equation still holds, and it reads
\begin{equation}
\begin{split} 
 &\int_0^{\infty}\!\!\! \int_{\RR^d} {\psi \big(z_{\nu}'(\pnu
)-z_{\nu}'(\wnu)\big)\frac{\pnu-\wnu}{\nu}}+\psi z_{\nu}''(\wnu)|\nabla \wnu|^2  \\[0.3em]
&\quad+\int_0^{\infty}\!\!\! \int_{\RR^d} e_{\nu}(\rhonu)(\nabla \wnu\cdot \nabla \psi -\partial_t \psi)-z_{\nu}(\wnu)\Delta \psi \\[0.5em]
= &\int_{\RR^d} \psi e_{\nu}(\rhonu^{\mathrm{in}})+\int_0^{\infty}\!\!\! \int_{\RR^d}  \psi e'_{\nu}(\rhonu) (\rhonu^{(1)} G^{(1)}(\pnu) +\rhonu^{(2)} G^{(2)}(\pnu)).
\end{split}
\end{equation}
The dissipative structure of the equation remains the same, therefore, provided that one can show the boundedness of the reaction term, the same implications of the energy evolution equation for the one-species case, for instance Lemma~\ref{lemma:w_bounds}, still hold. 
Thus, we get to the same conclusion as at the beginning of Section~\ref{sec: strong comp}, and the limit equation is 
$$\partial_t \rho = \nabla \cdot m +R,$$
where $R$ is now the weak limit of $\rhonu^{(1)}G^{(1)}(\pnu) +\rhonu^{(2)} G^{(2)}(\pnu)$.  Again the difference in the source term does not affect the arguments of Propositions \ref{prop:w_gradient_control} and \ref{prop:density_swap}, hence, it follows that $m=\rho\nabla p$.

Finally, to complete the rest of the argument in Section 5, the only place where we need to use the specific structure of the growth term is in showing that
\begin{equation*}
   \limsup_{\nu\to 0} \int_0^T\!\!\! \int_{\RR^d} t (\rhonu^{(1)}G^{(1)}(\pnu) +\rhonu^{(2)} G^{(2)}(p_\nu)) f^*_{\nu}(\pnu)\leq \int_0^T\!\!\! \int_{\RR^d} t R f^*_{0}(p).
\end{equation*}
This is true thanks to Lemma~\ref{lem:growth_weak_limit} applied to each term individually.
Hence, we can simply follow the arguments of Propositions \ref{prop:nablafstar} and \ref{prop:strong} to deduce the strong convergence of $\nabla \wnu$ so we are done.

\section{Convexity Lemmas}\label{app:1}

Here we give the proofs of Lemma~\ref{lem:z_from_e} and Lemma~\ref{lem:e_from_z}. 

\begin{proof}[Proof of Lemma~\ref{lem:z_from_e}]
    Note that $f_{\nu}^{*\prime}$ is nonnegative, bounded on compact subsets of $[0,\infty)$, and well defined almost everywhere.  Thus, the product $f^{*\prime}_{\nu}(\beta)S'(\beta)$ is well defined and nonnegative almost everywhere. Hence, it follows that $z_{\nu}$ is convex with $z_{\nu}\in W^{2,1}_{\loc}([0,\infty))$.

    It remains to prove that for all $a\in \dom(f_{\nu})$ and $b\in \partial f_{\nu}(a)$ there exists $c\in \partial e_{\nu}(a)$ such that 
    \[
    ac-e_{\nu}(a)=z'_{\nu}(b).
    \]
Let us choose $c=S(b)$.  Note that $e_{\nu}$ is differentiable whenever $f_{\nu}$ is differentiable, and at any point of differentiability $a$, we have  $e_{\nu}'(a)=S(f'_{\nu}(a))$.  Since $f_{\nu}$ is differentiable almost everywhere on $\dom(f_{\nu})$, it follows that 
\[
e_{\nu}(a)= \int_0^a S(f_{\nu}'(\alpha))d\alpha.
\] 

Now we want to show that the integral for $e_{\nu}$ is related to the integral for $z'_{\nu}$.
Define $b_{-}:=\inf \partial f_{\nu}(a)$. 
Using the fact that $f_{\nu}'$ pushes the Lebesgue measure on $[0,a]$ to the measure valued second derivative of $f_{\nu}^*$ on $[0,b_{-}]$, we have
\[
e_{\nu}(a)=\int_0^{b_{-}}S(\beta) df_{\nu}^*(\beta),
\]
where we write $df_{\nu}^*(\beta)$ to denote the measure valued second derivative of $f_{\nu}^*$.  After integrating by parts, we find that 
\[
e_{\nu}(a)=aS(b_{-})-\int_0^{b_{-}}S'(\beta) f_{\nu}^{*\prime}(\beta).
\]
Therefore, 
\[
e_{\nu}(a)+z'_{\nu}(b)=aS(b_{-})+\int_{b_{-}}^b f^{*\prime}_{\nu}(\beta)S'(\beta)\, d\beta.
\]
Since $f_{\nu}^{*\prime}(\beta)=a$ almost everywhere on $[b_{-},b]$, it follows that
\[
e_{\nu}(a)+z_{\nu}'(b)=aS(b),
\]
which was the desired result.
\end{proof}

\begin{proof}[Proof of Lemma~\ref{lem:e_from_z}]

Since $f_{\nu}'$ is well defined almost everywhere on $\dom(f_{\nu})$ and $z'_{\nu}$ is continuous, the integrand in \eqref{eq:e_from_z} is well-defined almost everywhere and measurable.  Hence, $e_{\nu}$ is well defined. 

Next, we check that $e_{\nu}$ is right-continuous at $0$, namely that $\liminf_{a\to 0^+} e_{\nu}(a)=e_{\nu}(0)=0$.  $z_{\nu}$ is convex and continuously differentiable with $z'_{\nu}(0)=0$ and $f_{\nu}$ satisfies $\lim_{a\to 0^+}\frac{f_{\nu}(a)}{a}=0$. As a result, given $\epsilon>0$, there exists $\delta\in (0,a_1)$ such that 
\[
|z'(b)|<\epsilon \quad \textup{for all} \; b\in \partial f(a), \; a\in [0,\delta].
\]
 Therefore, for $a\in [0,\delta]$, we have 
\[
|e_{\nu}(a)|\leq  a\int_{\delta}^{a_1} \frac{|z'_{\nu}(f'_{\nu}(\alpha))|}{\alpha^2}\, d\alpha+ a\epsilon \int_{a}^{\delta} \frac{1}{\alpha^2}\, d\alpha\leq \frac{a}{\delta^2} \int_{\delta}^{a_1} |z'_{\nu}(f'_{\nu}(\alpha))|\, d\alpha+ \epsilon. 
\]
Hence, $\liminf_{a\to 0^+} e_{\nu}(a)=0$.

Now let us show that $e_{\nu}$ is convex. Fix some $\delta>0$ and define $Z_{\delta}:\dom(f_{\nu})\to \RR$ by taking
\[
Z_{\delta}(a)=\essinf_{\alpha\in \dom(f_{\nu})} z'_{\nu}(f'_{\nu}(\alpha))+\frac{1}{2\delta}|a-\alpha|^2.
\]
It then follows that $Z_{\delta}$ is increasing and Lipschitz. Let 
\[
e_{\nu,\delta}(a)=a\int_{a_1}^a \frac{Z_{\delta}(\alpha)}{\alpha^2}d\alpha.
\]
Differentiating, and then integrating by parts, we see that 
\[
e_{\nu,\delta}'(a)=\frac{Z_{\delta}(a)}{a}+\int_{a_1}^a \frac{Z_{\delta}(\alpha)}{\alpha^2}d\alpha=\frac{Z_{\delta}(a_1)}{a_1}+\int_{a_1}^a \frac{Z_{\delta}'(\alpha)}{\alpha}d\alpha,
\]
and it is now clear that $e_{\nu,\delta}'$ is increasing and hence $e_{\nu,\delta}$ is convex.   Finally, since
\[
e_{\nu}(a)=\sup_{\delta>0} e_{\nu,\delta}(a),
\]
$e_{\nu}$ is convex.

Now let us show that $e_{\nu}$ satisfies the identity \eqref{eq:abc_equation}. If $0\in \dom(e_{\nu})$, then we can choose any $c\in \partial e_{\nu}(0)$ and the identity will hold, since $\partial f_{\nu}(0)=\{0\}$ and $z'(0)=0=c\cdot 0-e_{\nu}(0)$.   Otherwise, given $a\in \dom(e_{\nu})-\{0\}$ and $b\in \partial f_{\nu}(a)$, choose $c:=\frac{z'_{\nu}(b)}{a}+\frac{e_{\nu}(a)}{a}$.  Clearly with this choice, $ac-e_{\nu}(a)=z'_{\nu}(b)$, thus, we only need to check that $c\in \partial e_{\nu}(a)$. Since $e_{\nu}$ is continuous at 0, this amounts to showing that for any $\alpha\in \dom(f_{\nu})-\{0\}$, 
\[
e_{\nu}(\alpha)\geq e_{\nu}(a)+(\alpha-a)c.
\]
Multiplying both sides of the equation by $a$ and rearranging, we need 
\[
a e_{\nu}(\alpha)-\alpha e_{\nu}(a)\geq (\alpha-a)z'_{\nu}(b).
\]
Note that
\[
a e_{\nu}(\alpha)-\alpha e_{\nu}(a)=a \alpha \int_{a}^{\alpha} \frac{z'_{\nu}(f'_{\nu}(\alpha))}{\alpha^2}\, d\alpha.
\]
Since $z'$ and $f'$ are increasing, if $\alpha\geq a$, then 
   \[
a e_{\nu}(\alpha)-\alpha e_{\nu}(a)=a \alpha \int_{a}^{\alpha} \frac{z'_{\nu}(f'_{\nu}(\alpha))}{\alpha^2}d\alpha \geq a\alpha z'_{\nu}(b)\int_{a}^{\alpha} \frac{1}{\alpha^2}d\alpha=(\alpha-a)z'_{\nu}(b).
\]
The same logic also produces the case $\alpha\leq a$.

Finally, it remains to argue that $e_{\nu}$ is unique up to a linear factor.  Suppose that $e_{\nu}$ and $\tilde{e}_{\nu}$ are convex functions such that $e_{\nu}(0)=0=\tilde{e}_{\nu}(0)$ and both satisfy the coupling relation \eqref{eq:abc_equation} with $z_{\nu}$.  Choose a value $a>0$ where $e_{\nu}$ and $\tilde{e}_{\nu}$ are differentiable.  From the coupling relation, it then follows that
\[
a(e_{\nu}'(a)-\tilde{e}_{\nu}'(a))=e_{\nu}(a)-\tilde{e}_{\nu}(a).
\]
If we define $H(a):=e_{\nu}(a)-\tilde{e}_{\nu}(a)$, then at every point $a$ where $e_{\nu}, \tilde{e}_{\nu}$ are differentiable, we have
\[
H'(a)-\frac{1}{a}H(a)=0.
\]
Since $H$ is locally Lipschitz, this implies that there exists $\lambda\in \RR$ such that $H(a)=\lambda a$ for all $a\in \dom(f_{\nu})$.
\end{proof}

\section{Powers of the density and entropy dissipation}\label{app: eee powers}

We now show two other applications of the energy evolution equation \eqref{eq:eee} which provide the dissipation of standard quantities such as powers of the density and the Boltzmann entropy.
 
\begin{cor}[Powers of $\rhonu$ and entropy dissipation]\label{cor:powers}
If $f_{\nu}^*\in W^{2,1}(\RR)$, then for any $m>1$
\begin{equation}
    \label{eq:edi_power}
    \begin{split}
&\int_{\RR^d} \frac{\rhonu^{m}(T)-\rhonu(T)}{m-1} + \int_0^T \!\!\!\int_{\RR^d} m f_\nu^{*\prime}(\wnu)^{m-1} f^{*\prime\prime}_{\nu}(\wnu) |\nabla \wnu|^2 +(f^{*\prime}_{\nu}(\pnu)^m-f^{*\prime}_{\nu}(\wnu)^m)\frac{\pnu-\wnu}{\nu}\\[0.3em] &= \int_0^T\!\!\! \int_{\RR^d} \big(\rhonu^m+\frac{\rhonu^{m}-\rhonu}{m-1}\big) G(\pnu)+\int_{\RR^d} \frac{(\rhonu^{\mathrm{in}})^m -\rhonu^{\mathrm{in}}}{m-1}.
\end{split}
\end{equation}
Moreover, if $|x|^2 \rhonu^{\mathrm{in}}\in L^1(\RR^d)$, then we have
\begin{equation}
    \label{eq:edi entropy}
    \begin{split}
\int_{\RR^d} \rhonu(T) \ln \rhonu(T)-\rhonu(T)  &+ \int_0^T \!\!\!\int_{\RR^d} f^{*\prime\prime}_{\nu}(\wnu) |\nabla \wnu|^2 +(f^{*\prime}_\nu(\pnu)-f^{*\prime}_{\nu}(\wnu))\frac{\pnu-\wnu}{\nu}\\ = & \int_0^T\!\!\! \int_{\RR^d} \rhonu\ln \rhonu G(\pnu)+\int_{\RR^d} \rhonu^{\mathrm{in}} \ln \rhonu^{\mathrm{in}} -\rhonu^{\mathrm{in}} .
\end{split}
\end{equation}
\end{cor}
\begin{proof}
  First choose some $m>1$ and set $e_{\nu,m}(a)=\frac{a^{m}-a}{m-1}$ and $z'_{\nu,m}(b)=f_{\nu}^{*\prime}(b)^m$.  One can then check that $e_{\nu,m}, z_{\nu,m}$ are convex functions satisfying the coupling relation \eqref{eq:abc_equation} such that $\lim_{a\to 0}\frac{e_{\nu,m}(a)}{a}=-\frac{1}{m-1}$. Then, equation \eqref{eq:edi_power} now follows from \eqref{eq:edi_T} with $\eta=1$. 

To obtain \eqref{eq:edi entropy}, we want to send $m\to 1$ in \eqref{eq:edi_power}. It is easy to check that the uniform bound on the second moment of $\rhonu$ is propagated through time by computing $\frac{\mathrm{d}}{\dx t}\int |x|^2\rho_\nu$. Therefore, $\rhonu \ln \rhonu$ is bounded in $L^\infty([0,T]; L^1(\RR^d))$.  Hence, we can pass to the limit $m\to 1$ by dominated convergence.   After taking the limit, the precise form of \eqref{eq:edi entropy} follows from modifying the entropy by a linear factor from $\rhonu\ln(\rhonu)$ to $\rhonu\ln(\rhonu)-\rhonu$, which as we noted earlier does not affect the evolution equation.
\end{proof}
 
We point out that Corollary~\ref{cor:powers} includes the energy and entropy inequalities found in \cite{ES2023}  and 
\cite{DDMS}  with $f^{\star\prime}_\nu (b)= b^{1/m}$, for $m>1$ and $m=1$, respectively.

\bibliographystyle{abbrv}
\bibliography{literature}

\end{document}